\documentclass[reqno]{amsart}
\usepackage{amsmath}
\usepackage{amsthm}
\usepackage{amssymb}

\usepackage[colorlinks=true]{hyperref}
\hypersetup{urlcolor=blue, citecolor=red}

\newtheorem{Lemma}{Lemma}[section]
\newtheorem{Theorem}[Lemma]{Theorem}
\newtheorem{Proposition}[Lemma]{Proposition}

\newcommand{\Span}{\operatorname{span}}
\newcommand{\vol}{\operatorname{Vol}}

\newcommand{\diam}{\operatorname{diam}}
\newcommand{\de}[2]{\frac{\partial #1}{\partial #2}}
\newcommand{\p}{\partial}
\newcommand{\id}{\operatorname{Id}}
\newcommand{\supp}{\operatorname{supp}}
\newcommand{\sgn}{\operatorname{sgn}}

\begin{document}
\title[Inverse radiative transfer in the presence of a magnetic field]{An inverse radiative transfer in refractive media equipped with a magnetic field}
\author[Yernat M. Assylbekov]{Yernat M. Assylbekov}
\address{Department of Mathematics, University of Washington, Seattle, WA 98195-4350, USA}
\email{y\_assylbekov@yahoo.com}

\author[Yang Yang]{Yang Yang}
\address{Department of Mathematics, University of Washington, Seattle, WA 98195-4350, USA}
\email{yangy5@uw.edu}

\begin{abstract}
We study the reconstruction of the attenuation and absorption coefficients in a stationary linear transport equation from knowledge of albedo operator in dimension $n\geq 3$ on a Riemannian manifold in the presence of a magnetic field. We show the direct boundary value problem is well posed under two types of subcritical conditions. We obtain uniqueness and non-uniqueness results of the reconstruction under some restrictions. Finally, stability estimates are also established.
\end{abstract}
\maketitle

\section{Introduction}
In this article we study an inverse problem of radiative transfer on a Riemannian manifold in the presence of a magnetic field. More precisely, we consider the problem of recovering the absorption and scattering properties
of a medium from the measurements of the response to near-infrared light of the density of particles at the boundary. This problem is also called {\it\bfseries optical tomography} and has been applied
to the problems of medical imaging; see \cite{A} for more details. There are many other applications of the optical tomography such as diffused tomography \cite{SGKZ}, atmospheric remote sensing \cite{Bi}, nuclear physics \cite{MC}.

Mathematically, radiative transfer is governed by the fundamental law that describes the scattering of a low-density beam of particles from higher-density material. This law is essentially a balance equation that is called {\it\bfseries transport equation}. When the density of particles is low enough that interaction between particles may be neglected, the transport equation is called {\it\bfseries linear}. The earliest investigation of linear transport equation were carried out due to its appearance in the theory of radiative transfer, in a simple neutron diffusion situation, in the theory of plasmas, as well as in sound propagation theory. When the scattering process is independent of time, the transport equation is called {\it\bfseries stationary}. In this paper our main interest lies in the stationary linear transport equation.

In many cases one has to consider the transport equation in the presence of a magnetic field. For instance, the particles used in optical tomography may be charged during their migration and, with the existence of a magnetic field, experiencing a magnetic force. Similar situation occurs in the transport of charged plasmas when embedded into a magnetic filed. Therefore, it is necessary to study the above inverse radiative transfer problem with the existence of a magnetic field. In the following we shall consider this problem in a more general setting, that is, we shall consider the inverse problem of a radiative transfer for stationary linear transport equation on a Riemannian manifold in the presence of a magnetic field.

Now we describe the mathematical model of this problem. Suppose the medium is a compact $n$-dimensional smooth manifold $M$ with boundary endowed with Riemannian metric $g$. Let
$\pi:TM\to M$ denote the canonical projection,
$\pi:(x,\xi)\mapsto x$,
$x\in M$, $\xi\in T_xM$.
Denote by $\omega_0$ the canonical symplectic form on $TM$, which
is obtained by pulling back the canonical symplectic form of $T^*M$
via the Riemannian metric.
Let $H:TM\to\mathbb R$ be defined by
$$
H(x,\xi)=\frac 12|\xi|^2_{g(x)},\quad (x,\xi)\in TM.
$$
The Hamiltonian flow of $H$ with respect to $\omega_0$
gives rise to the geodesic flow of $(M,g)$.
Let $\Omega$ be a closed $2$-form on $M$ and consider
the new symplectic form $\omega$ defined as
$$
\omega=\omega_0+\pi^*\Omega.
$$
The Hamiltonian flow of $H$ with respect to $\omega$ gives rise to
the {\it\bfseries magnetic geodesic flow}
$\phi_t:TM\to TM$.
This flow models the motion of
a unit charge of unit mass in a magnetic field
whose {\it\bfseries Lorentz force} $Y:TM\to TM$
is the bundle map uniquely determined by
$$
\Omega_x(\xi,\eta)=\langle Y_x(\xi),\eta\rangle
$$
for all $x\in M$ and $\xi,\eta\in T_xM$. Every trajectory of the magnetic flow
is a curve on $M$ called {\it\bfseries magnetic geodesic}. Magnetic geodesics satisfy Newton's law of motion
\begin{equation}\label{m-geodesic}
\nabla_{\dot\gamma}\dot\gamma=Y(\dot\gamma).
\end{equation}
Here $\nabla$ is the Levy-Civita connection of $g$. The triple $(M,g,\Omega)$ is said to be a {\it\bfseries magnetic system}. In the absence of magneti field, that is $\Omega=0$, we recover the ordinary geodesic flow
and ordinary geodesics. Note that on the magnetic geodesics time is not reversible, unless $\Omega=0$; note also that to ensure each magnetic geodesic has unit tangent vectors, a specific energy level needs to be specified, from now upon we assume this has been done so that the tangent vectors to each magnetic geodesic is of unit length.

Magnetic flows were firstly considered in \cite{AS, Ar} and it was shown in \cite{ArG, Ko, N1, N2, NS, PP} that they are related to dynamical systems, symplectic geometry, classical mechanics and mathematical mechanics. The inverse problems related to magnetic flows were studied in \cite{Ain,DPSU,DU}

Now we consider stationary linear transport equation for magnetic system $(M,g,\Omega)$. Let $(x,\xi)$ is a point in the sphere bundle $SM$ and denote by $u(x,\xi)$ the density of particles at position $x$ with velocity $\xi$. In magnetic Riemannian manifold function $u$ satisfies the stationary linear transport equation
\begin{equation}\label{mag-tr-eq}
-\mathbf G_{\mu}u(x,\xi)-a(x,\xi)u(x,\xi)+\int_{S_x M}k(x,\eta,\xi)u(x,\eta)\,d\sigma_{x}(\eta)=0,
\end{equation}
where $d\sigma_{x}$ is the volume form on $S_x M$ determined by the metric $g$ at point $x$. Here and further by $\mathbf G_\mu$ we denote the generator of the magnetic flow.

Let us give the short explanation of the transport equation \eqref{mag-tr-eq}. The first term on the left models the motion of a particles, along the magnetic geodesics of $(M,g,\Omega)$, that does not interact with the material. The second term describes the loss of particles at $(x,\xi)$ while scattering to another direction or while absorption, quantified by function $a(x,\xi)$ which we call {\it\bfseries attenuation coefficient}. The third term describes production of particles at $(x,\xi)$ scattering from other directions. The kernel $k(x,\eta,\xi)$, called {\it\bfseries scattering coefficient}, represents the probability of a particle at $(x,\eta)$ scattering to $(x,\xi)$.

When the metric is Euclidean, then \eqref{mag-tr-eq} is a model for transport in a medium with constant index of refraction. The case of the metric conformal to the Euclidean corresponds to transport with varying, isotropic index of refraction. In this case \eqref{mag-tr-eq} can be derived as a limiting case of Maxwell's equations with isotropic permeability, if $\Omega=0$. For more details see e.g. \cite{B}. For general Riemannian metric, \eqref{mag-tr-eq} is considered as a model for transport in a medium with varying, anisotropic index of refraction.

Define the incoming and outgoing sphere bundles on $\p M$ as
$$
\p_{\pm}SM=\{(x,\xi)\in SM|x\in\p M,\pm\langle\xi,\nu(x)\rangle>0\},
$$
here $\nu$ is the inward unit normal vector to $\p M$. The medium is probed with a given radiation $u_+(x,\xi)$, which is a function on $\p_+SM$. Let $u$ be the solution to \eqref{mag-tr-eq} with boundary condition $u|_{\p_+SM}=u_+$, assuming it exists. Then we define the {\it\bfseries albedo operator}
$$
\mathcal A:u_+\to u|_{\p_-SM},
$$
which takes incoming flux $u_+$ to outgoing flux $\mathcal A(u_+)=u|_{\p_- SM}$ on the boundary. The goal is to recover the attenuation coefficient $a(x,\xi)$ and the scattering coefficient $k(x,\eta,\xi)$ from the knowledge of albedo operator $\mathcal{A}$.

There have been a lot of papers devoted to this problem, and we list some of these results here. In Euclidean case, uniqueness of this recovery has been considered under various assumptions in \cite{CS,MST,ST,SU}. Stability of this reconstruction process is studied in \cite{BJ,MST2}. This problem have also been investigated on manifolds. In \cite{M}, was considered the problem of reconstructing these two coefficients on a Riemannian manifold equipped with a known metric, and was proved uniqueness of the reconstruction process under some geometric assumptions. A stability result in Riemannian setting was later showed in \cite{MST3}.

In this paper we inquire the problem of reconstruction the attenuation coefficient $a(x,\xi)$ and the scattering coefficient $k(x,\eta,\xi)$ for all $x\in M$ and $\xi,\eta\in S_x M$, from albedo operator in the presence of a magnetic field. We also investigate the natural problem of determining the metric $g$ and the magnetic field $\Omega$ from $\mathcal A$.

The structure of this paper is as follows. In Section 2 we state our main results on uniqueness and non-uniqueness, as well as determination of the metric from albedo operator. In Section 3.1 we reduce well-posedness of the direct problem to an integral equation, and in Section 3.2 we show that the direct problem is well posed under two types of subcritical conditions. This makes albedo operator $\mathcal{A}$ a bounded operator between appropriate function spaces. Section 4 is devoted to computing the distribution kernel of albedo operator. In Section 5 we extract useful expressions from this distributional kernel. Section 6 contains the proofs of uniqueness and non-uniqueness results in recovering the attenuation coefficient $a$ and the scattering coefficient $k$. In Section 7 we we state our main result on stability. Section~8 is devoted to preparatory estimates, and Section~9 is devoted to the proof of the stability result.

\noindent{\bf Acknowledgements.} The first author would like to express his acknowledgements
to Prof. Nurlan S. Dairbekov for all his help and encouragement. Both authors
would like to express their acknowledgements to Prof. Gunther Uhlmann for constant
assistance and support. The work of both authors was partially supported by NSF.

\section{Statement of the main results on uniqueness and non-uniqueness}
Let $(M,g,\Omega)$ be a magnetic system. For a given $(x,\xi)\in SM$, we define $\tau_+(x,\xi)$ (resp. $\tau_-(x,\xi)$) to be the positive (resp. negative) time when the unit speed magnetic geodesic $\gamma_{x,\xi}$, with $\gamma_{x,\xi}(0)=x$ and $\dot\gamma_{x,\xi}(0)=\xi$, exits the manifold $M$. Also, define $\tau=\tau_+-\tau_-$. In general, $\tau_\pm$ may be infinite. So, we assume that the magnetic system $(M,g,\Omega)$ is {\it\bfseries simple}, which ensures that $\tau_\pm$ are well-defined, finite, and moreover, continuous on $SM$. By simplicity of the magnetic system we mean that
\begin{enumerate}
\item[1.] The boundary $\p M$ is strictly magnetic convex.
\item[2.] The exponential map $\exp_x^\mu:(\exp_x^\mu)^{-1}M\to M$ is a $C^1$-diffeomorphism for every $x\in M$.
\end{enumerate}
In this case, $M$ is diffeomorphic to the unit ball of $\mathbb R^n$ and therefore $\Omega$ is exact i.e., is of the form
$$
\Omega=dm
$$
where $m$ is a 1-form on $M$ --- {\it\bfseries magnetic potential}. We call $(M,g,m)$ a {\it\bfseries simple magnetic system} on $M$.. The notion of simplicity arises naturally in the context of the boundary rigidity problem \cite{DPSU,Mi}. Since $\tau_+$ and $\tau_-$ are finite functions, we define the following function on $SM$
$$
\tau(x,\xi)=\tau_+(x,\xi)-\tau_-(x,\xi),\quad(x,\xi)\in SM.
$$
We also introduce the following volume form on $SM$
$$
d\Sigma^{2n-1}(x,\xi)=d\sigma_x(\xi)\,d\vol_g(x),\quad (x,\xi)\in SM
$$
where $d\vol_g$ is the volume form on $M$ determined by the metric $g$.

The pair $(a,k)$ is said to be {\it\bfseries admissible} if
$$
\begin{cases}0\le a \in L^\infty(SM),\\
0 \le k(x,\eta,\cdot) \in L^1 (S_xM), \text{ for a.e. }(x,\eta)\in SM,\\
\sigma_p(x,\eta) = \int_{S_x M}k(x,\eta,\xi)\,d\sigma_x(\xi)\in L^\infty(SM).
\end{cases}
$$
Admissibility condition guarantees that the second and third terms in \eqref{mag-tr-eq} are bounded operators. But the first term is unbounded. The direct problem \eqref{mag-tr-eq} with $u|_{\p_+SM}=u_+$ can be reduced to non-homogeneous equation with homogeneous boundary condition. Then we say that the direct problem \eqref{mag-tr-eq} with $u|_{\p_+SM}=u_+$ is {\it\bfseries well posed}, if it has a unique solution.

In particular, the direct problem is well posed if either one of the following two {\it\bfseries subcritical} conditions hold:
\begin{equation}\label{subcritical1}
\|\tau\sigma_p\|_{L^{\infty}(SM)}<1,
\end{equation}
\begin{equation}\label{subcritical2}
a(x,\xi)-\sigma_p(x,\xi)\ge0\text{ for a.e. }(x,\xi)\in SM,
\end{equation}
see Proposition~\ref{well-posed-subcritical1} and Proposition~\ref{well-posed-subcritical2}. These results generalize previous results in \cite{BJ,CS,DL,M,MK,RS}.

Even for the case of Euclidean space, albedo operator does not uniquely determine $a$ without further restriction on $a$; see \cite{CS}. Our first main result says that, under the assumption that the attenuation coefficient $a=a(x)$ is isotropic, that albedo operator determines admissible pair $(a,k)$. This result generalizes corresponding results obtained in \cite{CS,M}
\begin{Theorem}\label{thA}
Let $(M,g,m)$ be a simple magnetic system and $\dim M\geq 3$. Assume that the pair $(a,k)$ with isotropic attenuation coefficient is admissible, and that the direct problems is well posed. Then the corresponding albedo operator $\mathcal A$ determines the pair $(a,k)$ uniquely.
\end{Theorem}

The case when $a$ is anisotropic was considered in \cite{ST} for Euclidean space. It was shown that knowing the albedo operator, one can uniquely recover the class of ``gauge equivalent" admissible pairs $(a,k)$. This was extended to Riemannian manifolds in \cite{MST}.

Now, let us recall what gauge equivalence means. The notion of gauge equivalence, stated in \cite{ST} and extended to Riemannian manifolds in \cite{MST}, is valid in the presence of a magnetic field. Let $w\in L^\infty(SM)$ be positive with $1/w\in L^\infty(SM)$, $\mathbf G_\mu w\in L^\infty(SM)$ and such that $w|_{\p SM}=1$. Set
\begin{equation}\label{gauge-eq}
\tilde a(x,\xi)=a(x,\xi)-\mathbf G_\mu\log w(x,\xi),\quad \tilde k(x,\eta,\xi)=\frac{w(x,\xi)k(x,\eta,\xi)}{w(x,\eta)}.
\end{equation}
Then $u$ satisfies \eqref{mag-tr-eq} if and only if $\tilde u=w u$ satisfies
$$
-\mathbf G_{\mu}\tilde u(x,\xi)-\tilde a(x,\xi)\tilde u(x,\xi)+\int_{S_x M}\tilde k(x,\eta,\xi)\tilde u(x,\eta)\,d\sigma_x(\eta)=0.
$$
Note that the albedo operator $\mathcal A$ for $(a,k)$ is equal to the albedo operator $\tilde{\mathcal A}$ for $(\tilde a,\tilde k)$, since $u=\tilde u$ on $\p(SM)$.

We say that two admissible pairs $(a,k)$ and $(\tilde a,\tilde k)$ are {\it\bfseries gauge equivalent} if there exists a positive function $w\in L^\infty(SM)$ with $1/w\in L^\infty(SM)$, $\mathbf G_\mu w\in L^\infty(SM)$ and $w|_{\p SM}=1$ such that \eqref{gauge-eq} holds. We denote this relation (which is an equivalence relation) by $(a,k)\sim(\tilde a,\tilde k)$.

In this paper we prove the following result which is a generalization of corresponding results in \cite{MST,ST}

\begin{Theorem}\label{thB}
Let $(M,g,m)$ be a simple magnetic system with $\dim M\geq 3$. Assume that the pairs $(a,k)$ and $(\tilde a,\tilde k)$ are admissible, and that their corresponding direct problems are well posed. Then $\mathcal A=\tilde{\mathcal A}$ if and only if $(a,k)\sim(\tilde a,\tilde k)$.
\end{Theorem}

Now, assume that the attenuation coefficient depends only on the line determined by vector $\xi$, that is
\begin{equation}\label{a-symm}
a(x,\xi)=a(x,-\xi),\quad x\in M,\quad\xi\in S_x M,
\end{equation}
and assume that the scattering coefficient is positive and symmetric in the following sense:
\begin{equation}\label{k-symm}
k(x,\eta,\xi)=k(x,\xi,\eta),\quad x\in M,\quad\xi,\eta\in S_x M.
\end{equation}
In this case we prove the following uniqueness result generalizing corresponding results obtained in \cite{MST,ST}.
\begin{Theorem}\label{thC}
Let $(M,g,m)$ be a simple magnetic system and $\dim M\geq 3$. Let $(a,k)$ and $(\tilde a,\tilde k)$ be admissible pairs, and assume either \eqref{subcritical1} or \eqref{subcritical2} holds.
\begin{itemize}
\item[(a)] Suppose that $k,\tilde k>0$ satisfy \eqref{k-symm}, then $\mathcal A=\tilde{\mathcal A}$ if and only if $k=\tilde k$ and  $a=\tilde a+\mathbf G_\mu h$ for some function $h$ on $M$ vanishing on $\p M$.
\item[(b)] Suppose that $k,\tilde k>0$ satisfy \eqref{k-symm} and $a,\tilde a$ satisfy \eqref{a-symm}, then $\mathcal A=\tilde{\mathcal A}$ if and only if $(a,k)=(\tilde a,\tilde k)$.
\end{itemize}
\end{Theorem}

The example of a situation where symmetry \eqref{k-symm} holds is when scattering coefficient depends on the angle of scattering, that is, $k(x,\xi,\eta)=k(x,\theta(\xi,\eta))$ where $\theta(\xi,\eta)$ is the angle between $\xi$ and $\eta$ at $x$.

In \cite{ST}, explicit reconstruction formulas were given if the pair $(a,k)$ satisfies symmetry and continuity assumptions in Euclidean case. For Riemannian setting these formulas were derived in \cite{MST}. The proof was based on the reversibility property of geodesic flow (in Riemannian case), which does not hold for magnetic flow. In general, magnetic flows are not reversible. It would be interesting if one derives explicit formulas to express $(a,k)$ under symmetry assumption in the presence of a magnetic field.

\section{Preliminaries and the direct problems}

In this section we show well posedness of the direct problem \eqref{mag-tr-eq} under two subcritical conditions \eqref{subcritical1} and \eqref{subcritical2} respectively. As a result, albedo operator becomes a well-defined bounded linear operator between appropriate function spaces.

\subsection{Preliminaries}\label{3.1}
We introduce measure $\,d\mu$ on $\p_\pm SM$. Let $\,d\Sigma^{2n-1}$ the volume form on $\p_\pm SM$ obtained by the natural restriction of  $\,d\Sigma^{2n-1}$ to $\p_\pm SM$. We define
$$
d\mu(x',\xi')=|\langle \xi',\nu(x')\rangle|\,d\Sigma^{2n-2}(x',\xi').
$$
We denote by $L^1(\p_\pm SM)$ the corresponding $L^1$ spaces.
\begin{Proposition}\label{Santalo}
If $f \in L^1(SM)$ then
\begin{align*}
\int_{SM} f(x,\xi)d\Sigma^{2n-1}(x,\xi)&=\int_{\p_+ SM}\int_0^{\tau_+(x',\xi')}f(\phi_t(x',\xi'))dtd\mu(x',\xi')\\
&=\int_{\p_- SM}\int^0_{\tau_-(x',\xi')}f(\phi_t(x',\xi'))dtd\mu(x',\xi').
\end{align*}
\end{Proposition}
\begin{proof}
The first equality is proved in \cite[Lemma~A.8]{DPSU}. The proof of the second equality is similar.
\end{proof}

We proceed now to solve the boundary value problem \eqref{mag-tr-eq} with
$u|_{\p_+ SM} = u_+ \in L^1(\p_+ SM)$ and shall do so by reformulating the problem as an integral equation. First,
$$
\mathbf G_{\mu} u(x,\xi)=g(x,\xi),  \qquad u|_{\p_+ SM} = 0
$$
has solution
$$
u(x,\xi)=\int_{\tau_-(x,\xi)}^0 g(\phi_s(x,\xi))ds.
$$

Next, let $u_+$ be a function on $\p_+ SM$. Then
$$
\mathbf G_{\mu} u(x,\xi)+a(x,\xi)u(x,\xi) = 0\text{ in }SM,\quad u|_{\p_+ SM}=u_+
$$
has solution $Ju_+$, where
$$
Ju_+(x,\xi) = E(x,\xi,\tau_-(x,\xi),0)u_+(\phi_{\tau_-(x,\xi)}(x,\xi))
$$
and $E(x,\xi,s,t) =\exp{\left(-\int\limits_s^t a (\phi_p(x,\xi))\,dp \right)}$.

\begin{Proposition}\label{J is bounded}
If $u_+\in L^1(\p_+ SM)$, then $Ju_+\in L^1(SM,\tau^{-1}\, d\Sigma^{2n-1})$ and
$$
 \|\tau^{-1} Ju_+\|_{L^1(SM)}\le\|u_+\|_{L^1(\p_+ SM)}.
$$
\end{Proposition}
\begin{proof}
Using Proposition \ref{Santalo},
\begin{multline*}
 \|\tau^{-1} Ju_+\|_{L^1(SM)}\\
\le\int_{\p_+ SM}\int_0^{\tau_+(x',\xi')}\frac{1}{\tau(\phi_t(x',\xi'))}|u_+(\phi_{\tau_-(\phi_t(x',\xi'))}(\phi_t(x',\xi')))|\,dt\,d\mu(x',\xi')\\
=\int_{\p_+ SM}\int_0^{\tau_+(x',\xi')}\frac{1}{\tau_+(x',\xi')}|u_+(x',\xi')|\,dt\,d\mu(x',\xi')\\
=\int_{\p_+ SM}|u_+(x',\xi')|\,d\mu(x',\xi')=\|u_+\|_{L^1(\p_+ SM)}.
\end{multline*}
This implies the statement of the proposition.
\end{proof}

Set $T_0u=-\mathbf G_\mu u-au$ in the distributional sense, and
$$
T_1u(x,\xi)=\displaystyle\int_{S_xM}k(x,\xi',\xi)u(x,\xi')\,d\sigma_{x}(\xi'),\quad T=T_0+T_1
$$
Sometimes we regard $T_{0}$ and $T$ as unbounded operators, in this case we write
$$
\mathbf{T}_{0}f=T_{0}f,\quad \mathbf{T}f=Tf,
$$
with domain
$$
D(\mathbf{T_{0}})=D(\mathbf{T})=\{f\in L^{1}(SM):\mathbf{G}_{\mu}f\in L^{1}(SM),\,f|_{\partial_{+}SM}=0\}.
$$
Then $\mathbf{T}_{0}:D(\mathbf{T}_{0}):\rightarrow L^{1}(SM)$ is closed, one-to-one and onto with a bounded inverse given by
$$
\mathbf{T}^{-1}_{0}f(x,\xi):=-\displaystyle\int^{0}_{\tau_{-}(x,\xi)}E(x,\xi,t,0)f(\phi_{t}(x,\xi))\,dt.
$$

Next we are going to reduce the equation \eqref{mag-tr-eq} to an integral equation.
Equation \eqref{mag-tr-eq} can be rewritten as $T_0 u+T_1 u=0$ and we solve it with $u|_{\p_+ SM} = u_+$.
Multiply $T_0 u+T_1 u=0$ by $E(x,\xi,t,0)$ and integrate it along the magnetic geodesic $\gamma_{x,\xi}$ from $\tau_-(x,\xi)$ to $0$. Taking to account the boundary condition $u|_{\p_+ SM}=u_+$, we get
\begin{equation}\label{part}
(\id+K)u=Ju_+
\end{equation}
where $K$ is the following integral operator
$$
Ku(x,\xi)=\mathbf{T}^{-1}_{0}T_{1}=-\int^0_{\tau_-(x,\xi)}E(x,\xi,t,0)(T_1 u)(\phi_t(x,\xi)))\,dt.
$$
Therefore, since $Ju_+\in L^1(SM,\tau^{-1}\, d\Sigma^{2n-1})$, solving the boundary value problem is equivalent to inverting the operator $\id+K$ on the weighted space $L^1(SM,\tau^{-1}\, d\Sigma^{2n-1})$.

Next step is to show that the well-posedness of the direct problem implies the boundedness of $\tau^{-1}\mathbf T^{-1}$ on $L^1(SM)$.

\begin{Proposition}\label{prop}
Let $(M,g,m)$ be a simple magnetic system. Assume that the pair $(a,k)$  is admissible and that the direct problems is well posed. Then the operator $\tau^{-1}\mathbf T^{-1}$ is bounded on $L^1(TM)$.
\end{Proposition}
\begin{proof}
Recall that solving \eqref{mag-tr-eq} with $u|_{\p_+ SM}\in L^1(\p_+ SM)$ is equivalent to solving \eqref{part}. Proposition~\ref{J is bounded} we have $Ju_+\in L^1(SM,\tau^{-1}\,d\Sigma^{2n-1})$. By hypothesis, the inverse of $\id+K$ exists as a bounded operator on $L^1(SM,\tau^{-1}\,d\Sigma^{2n-1})$. Therefore, writing
$$
\mathbf T^{-1}=(\id+K)^{-1}\mathbf T_0^{-1}:L^1(SM)\to L^1(SM,\tau^{-1}\,d\Sigma^{2n-1}),
$$
we see that $\tau^{-1}\mathbf T^{-1}$ is bounded on $L^1(SM)$.
\end{proof}
We need the following two lemmas studying boundedness of two related operators.
\begin{Lemma}\label{T_0-inverse}
The operator $\tau^{-1}\mathbf T_0^{-1}$ is bounded on $L^1(SM)$ with
$$
\|\tau^{-1}\mathbf T^{-1}_{0}f\|_{L^{1}(SM)}\leq \|f\|_{L^{1}(SM)}.
$$
\end{Lemma}
\begin{proof}
For $f\in L^1(SM)$,
\begin{multline*}
 \| \tau^{-1} \mathbf T_0^{-1} f\|_{L^1(SM)}\\
=\int_{\p_+ SM}\int_0^{\tau_+(x',\xi')}\Bigg|\int^0_{\tau_-(\phi_{t}(x',\xi')}E(\phi_t(x',\xi'),s,0) f(\phi_{t+s}(x',\xi'))\,ds\Bigg|\\
\times\displaystyle\frac{1}{\tau(\phi_t(x,\xi))}\,dt\,d\mu(x',\xi')\\
\le\int_{\p_+ SM}\int_0^{\tau_+(x',\xi')}\frac{1}{\tau(\phi_t(x,\xi))}\Bigg|\int^0_{-t}
f(\phi_{t+s}(x',\xi'))\,ds\Bigg|\,dt\,d\mu(x',\xi')\\
\le\int_{\p_+ SM}\int_0^{\tau_+(x',\xi')}\frac{1}{\tau_+(x',\xi')}\int_0^{\tau_+(x',\xi')}\big|f(\phi_s(x',\xi'))\big|\,ds\,dt\,d\mu(x',\xi')\\
= \displaystyle\int_{\p_+ SM}\int_0^{\tau_+(x',\xi')}|f(\phi_s(x',\xi'))|\,ds\,d\mu(x',\xi')=\|f\|_{L^1(SM)}.
\end{multline*}
We are done.
\end{proof}

\begin{Lemma}\label{T1}
The operator $T_{1}\tau$ is bounded on $L^{1}(SM)$ with
$$
\|T_{1}\tau f\|_{L^{1}(SM)}\leq\|\tau\sigma_{p}\|_{L^{\infty}(SM)}\|f\|_{L^{1}(SM)}.
$$
\end{Lemma}
\begin{proof}
For $f\in L^{1}(SM)$,
\begin{multline*}
\|T_{1}\tau f\|_{L^{1}(SM)}=\int_{M}\int_{S_{x}M}\left|\int_{S_{x}M}k(x,\xi',\xi)\tau(x,\xi')f(x,\xi')\,d\sigma_x(\xi')\right|\,d\sigma_x(\xi)\,d\vol_g(x)\\
\le\displaystyle\int_{M}\int_{S_{x}M}\int_{S_{x}M}k(x,\xi',\xi)\tau(x,\xi')|f(x,\xi')|\,d\sigma_x(\xi') \,d\sigma_x(\xi)\,d\vol_g(x)\\
=\int_{M}\int_{S_{x}M}\sigma_{p}(x,\xi')\tau(x,\xi')|f(x,\xi')|\,d\sigma_x(\xi')\,d\vol_g(x)\\
\le \|\tau\sigma_{p}\|_{L^{\infty}(SM)}\|f\|_{L^{1}(SM)}.
\end{multline*}
The proof is finished.
\end{proof}

\subsection{Well-posedness under subcritical conditions}
We define the function space $\mathcal{W}$ as follows
$$
\mathcal{W} = \{ f\in L^{1}(SM) : \mathbf G_{\mu} f \in L^1(SM), \tau^{-1}f \in L^1(SM)\},
$$
with norm
$$
\|f\|_{\mathcal{W}} = \|\mathbf G_{\mu} f\|_{L^1(SM)} + \|\tau^{-1}f\|_{L^1(SM)}.
$$
The following theorem is proven in \cite{CS} for Euclidean metric and in \cite{M} for the case of Riemannian metric, it enables us to define the trace of a function in $\mathcal{W}$.

\begin{Proposition}\label{th:norm}
If $f(x,\xi)\in \mathcal W$, then $\|f|_{\p_\pm SM}\|_{L^1(\p_\pm SM)}\leq\|f\|_{\mathcal{W}}$.
\end{Proposition}
\begin{proof}
It was show in \cite{CS} that if $h,\p_t h\in L^1([0,a])$, for some $a>0$, then the following inequality is valid
\begin{equation}\label{ineq1}
|h(0)|\le\|\p_t h\|_{L^1([0,a])}+\frac{1}{a}\|h\|_{|L^1([0,a])}.
\end{equation}
Similarly it can be easily shown that if $h,\p_t h\in L^1([b,0])$, for some $b<0$, then the following inequality is valid
\begin{equation}\label{ineq2}
|h(0)|\le\|\p_t h\|_{L^1([0,a])}-\frac{1}{b}\|h\|_{|L^1([b,0])}.
\end{equation}

Let $f(x,\xi)\in \mathcal W$. Define
$$
h(x',\xi',t)=f(\phi_t(x',\xi')), \quad (x',\xi')\in\p_+ SM,\quad 0\le t\le \tau_+(x',\xi').
$$
It is well known that $\mathbf G_{\mu}=\p\phi_t(x,\xi)/\p t|_{t=0}$.
Then
$$
\mathbf G_{\mu}f(\phi_t(x',\xi'))=\de{}{s}\Big|_{s=0}f(\phi_{t+s}(x',\xi'))=\p_t h(x',\xi',t),
$$
so $\p_t h\in L^1(SM)$. Therefore by Fubini's theorem, $\p_t h\in L^1([0,\tau_+(x',\xi')])$ for almost every $(x',\xi')$ as well as $h\in L^1([0,\tau_+(x',\xi')])$ for almost every $(x',\xi')$. Making use of \eqref{ineq1} we get
\begin{multline*}
|f(x',\xi')|=|f(\phi_0(x',\xi'))|\\
\le\int_0^{\tau_+(x',\xi')}|\mathbf G_{\mu}f(\phi_t(x',\xi'))|\,dt+\frac{1}{\tau_+(x',\xi')}\displaystyle\int_0^{\tau_+(x',\xi')}|f(\phi_t(x',\xi'))|\,dt.
\end{multline*}
Integrating this inequality over $\p_+ SM$ and applying Proposition~\ref{Santalo} to the right-hand side, we finish the proof of theorem for $f|_{\p_+ SM}$. The proof for $f|_{\p_- SM}$ can be done in similar way using \eqref{ineq2}.
\end{proof}

The operator $J$ is a bounded linear operator mapping $L^{1}(\p_{+}SM)$ into $\mathcal{W}$. This is the following proposition.
\begin{Proposition}\label{J}
The operator $J:L^1(\p_+ SM)\to\mathcal W$ is bounded with
$$
\|Ju_+\|_{\mathcal W} \le (1+\|\tau a\|_{L^\infty(SM)})\|u_+\|_{L^1(\p_+ SM)}.
$$
\end{Proposition}

\begin{proof}
Since $-\mathbf G_{\mu}Ju_+=aJu_+$, the following holds
$$
\|Ju_+\|_{\mathcal W}=\|\tau^{-1}Ju_+\|_{L^1(SM)}+\|a Ju_+\|_{L^1(SM)}\le(1+\|\tau a\|_{L^{\infty}(SM)})\|\tau^{-1}Ju_+\|_{L^1(SM)}.
$$
The proof is complete.
\end{proof}

Now, we prove that under subcritical condition \eqref{subcritical1} the direct problem is well posed.
\begin{Proposition}\label{well-posed-subcritical1}
Assume that the subcritical condition \eqref{subcritical1} holds.
\begin{itemize}
\item[(a)] $K=\mathbf T_0^{-1}T_1$ is bounded on $L^1(SM,\tau^{-1}d\Sigma^{2n-1})$,
the operator norm of $K$ is bounded by $\|\tau\sigma_p\|_{L^\infty(SM)}<1$, and so $\id+K$ is invertible on this space. Therefore, the direct problem \eqref{mag-tr-eq} with $u|_{\p_+ SM}=u_+\in L^1(\p_+ SM)$, is uniquely solvable with solution $u\in\mathcal {W}$.
\item[(b)] The albedo operator $\mathcal A : L^1(\p_+ SM)\to L^1(\p_- SM)$ is bounded map.
\end{itemize}
\end{Proposition}
\begin{proof}
As $K=\mathbf{T}^{-1}_{0}T_{1}$, from Lemma~\ref{T_0-inverse} and Lemma~\ref{T1} we have
\begin{multline*}
\|\tau^{-1}Kf\|_{L^{1}(SM)}=\|\tau^{-1}\mathbf{T}^{-1}_{0}T_{1}f\|_{L^{1}(SM)}\leq \|T_{1}f\|_{L^{1}(SM)}\\
\le \|\tau\sigma_{p}\|_{L^{\infty}(SM)}\|\tau^{-1}f\|_{L^{1}(SM)}.
\end{multline*}
The subcritical condition \eqref{subcritical1} is hence imposed so that $K$ is a bounded operator on $L^{1}(SM,\tau^{-1}d\Sigma^{2n-1})$ with norm strictly less than one, hence $\id+K$ is invertible on this space. Therefore \eqref{part}, and hence \eqref{mag-tr-eq}, with $u|_{\p_+ SM}=u_+$ has the unique solution $u=(\id+K)^{-1}Ju_+$. To finish the proof of (a) it is left to show that $u\in \mathcal W$ whenever $u_+\in L^1(\p_+ SM)$. Using Proposition~\ref{J is bounded}, we have
$$
\|\tau^{-1}u\|_{L^1(SM)}\le C\|u_+\|_{L^1(\p_+ SM)},
$$
with $C=\|(\id+K)^{-1}\|_{L^1(SM,\tau^{-1}\,d\Sigma^{2n-1})\to L^1(SM,\tau^{-1}\,d\Sigma^{2n-1})}$.

Since $\mathbf G_\mu u=-au+T_1u,$ we get
$$
\|\mathbf G_\mu u\|_{L^1(SM)}\le(\|\tau a\|_{L^{\infty}(SM)}+\|\tau\sigma_p\|_{L^{\infty}(SM)})\|\tau^{-1}u\|_{L^1(SM)}<\infty
$$
from the previous estimate. Thus $u\in \mathcal W$ and the proof of (a) is complete.

To prove item~(b), use Proposition~\ref{th:norm} for $\mathcal Au_+=u|_{\p_-SM}$ with previous two inequalities for $\tau^{-1}u$ and $\mathbf G_\mu u$ to get
\begin{multline*}
\|\mathcal Au\|_{L^1(\p_- SM)}\le\|\mathbf G_\mu u\|_{L^1(SM)}+\|\tau^{-1}u\|_{L^1(SM)}\\
\le (1+\|\tau a\|_{L^\infty(SM)}+\|\tau\sigma_p\|_{L^\infty(SM)})\|\tau^{-1}u\|_{L^1(SM)}\\
\le (1+\|\tau a\|_{L^\infty(SM)}+\|\tau\sigma_p\|_{L^\infty(SM)})(1-\|\tau\sigma_p\|_{L^\infty(SM)})^{-1}
\|u_+\|_{L^1(\p_+ SM)}.
\end{multline*}
This completes the proof of item (b).
\end{proof}

Next, we prove that under subcritical condition \eqref{subcritical2} the direct problem is well posed.
\begin{Proposition}\label{well-posed-subcritical2}
Assume that the subcritical condition \eqref{subcritical2} holds.
\begin{itemize}
\item[(a)] The operator $\id+K$ is invertible on $L^1(SM,\tau^{-1}\,d\Sigma^{2n-1})$ with $(\id+K)^{-1}=\id-\mathbf T^{-1}T_1$. Therefore, the forward problem \eqref{mag-tr-eq}, with $u|_{\p_+ SM}=u_+\in L^1(\p_+ SM)$, is uniquely solvable with solution $u\in{\mathcal W}$.
\item[(b)] The albedo operator $\mathcal A : L^1(\p_+ SM)\to L^1(\p_- SM)$ is bounded map.
\end{itemize}
\end{Proposition}
\begin{proof}
To show the invertibility of $\id+K$ on $L^1(TM,\tau^{-1}dV^{2n})$, we proceed in a different way. Consider the operator $T_{1}\mathbf{T}^{-1}_{0}:L^{1}(SM)\rightarrow L^{1}(SM)$, this is a bounded operator by Lemma \ref{T_0-inverse} and Lemma \ref{T1}. First, we will show that, assuming the subcritical condition~\eqref{subcritical2}, this operator has norm strictly less than one. For this purpose, we compute
\begin{multline*}
\|T_{1}\mathbf{T}^{-1}_{0}f\|_{L^{1}(SM)}\le \int_{M}\int_{S_{x}M}\int_{S_{x}M}\int^{0}_{\tau_{-}(x,\eta)}k(x,\eta,\xi)E(x,\eta,t,0)\\
\times|f(\phi_{t}(x,\eta))|\,dt\,d\sigma_x(\eta)\,d\sigma_x(\xi)\,d\vol_g(x)\\
=\int_{M}\int_{S_{x}M}\int^{0}_{\tau_{-}(x,\eta)}\sigma_{p}(x,\eta)E(x,\eta,t,0)|f(\phi_{t}(x,\eta))|\,dt\,d\sigma_x(\eta)\,d\vol_g(x)\\
\le\int_{M}\int_{S_{x}M}\int^{0}_{\tau_{-}(x,\eta)}\sigma_{p}(x,\eta)
e^{\int^{t}_{0}\sigma_{p}(\phi_{u}(x,\eta))\,du}|f(\phi_{t}(x,\eta))|\,dt\,d\sigma_x(\eta)\,d\vol_g(x).
\end{multline*}
Make a change of variable $(x,\eta)=\phi_{s}(x',\eta')$ on the right-hand side to get
\begin{multline*}
\|T_{1}\mathbf{T}^{-1}_{0}f\|_{L^{1}(SM)}\le \int_{\partial_{+}SM}\int^{\tau_{+}(x',\eta')}_{0}\int^{0}_{-s}\sigma_{p}(\phi_{s}(x',\eta'))
e^{\int^{t}_{0}\sigma_{p}(\phi_{u+s}(x',\eta'))\,du}\\
\times|f(\phi_{t+s}(x',\eta'))|\,dt\,ds\,d\mu(x',\eta')\\
=\int_{\partial_{+}SM}\int^{\tau_{+}(x',\eta')}_{0}\int^{s}_{0}\sigma_{p}(\phi_{s}(x',\eta'))
e^{\int^{t}_{s}\sigma_{p}(\phi_{u}(x',\eta'))\,du}\\
\times|f(\phi_{t}(x',\eta'))|\,dt\,ds\,d\mu(x',\eta').
\end{multline*}
Observe that
$$
\sigma_{p}(\phi_{s}(x',\eta'))e^{\int^{t}_{s}\sigma_{p}(\phi_{u}(x',\eta'))\,du}=-\displaystyle\frac{d}{ds}
\left(e^{\int^{t}_{s}\sigma_{p}(\phi_{u}(x',\eta'))\,du}\right).
$$
Switching the order of integration with respect to $s$ and $t$ we obtain
\begin{multline*}
\|T_{1}\mathbf{T}^{-1}_{0}f\|_{L^{1}(SM)}\\
\le\int_{\partial_{+}SM}\int^{\tau_{+}(x',\eta')}_{0}\bigg(\int^{\tau_{+}(x',\eta')}_{t}\sigma_{p}(\phi_{s}(x',\eta'))e^{\int^{t}_{s}\sigma_{p}(\phi_{u}(x',\eta'))\,du}\,ds\bigg)\\
\times|f(\phi_{t}(x',\eta'))|\,dt\,d\mu(x',\eta')\\
=\int_{\partial_{+}SM}\int^{\tau_{+}(x',\eta')}_{0}
\left(1-e^{-\int^{\tau_{+}(x',\eta')}_{t}\sigma_{p}(\phi_{u}(x',\eta'))\,du}\right)\\
\times|f(\phi_{t}(x',\eta'))|\,dt\,d\mu(x',\eta')\\
\le \left(1-e^{-2\|\sigma_{p}\|_{L^{\infty}(SM)}\|\tau\|_{L^{\infty}(SM)}}\right)\|f\|_{L^{1}(SM)}<\|f\|_{L^{1}(SM)}.
\end{multline*}
Therefore, $\id\pm T_1\mathbf T^{-1}_0$ are invertible on $L^1(SM)$. Now, we derive the invertibility of $\id+K$ on $L^{1}(SM,\tau^{-1}d\Sigma^{2n-1})$. Firstly, $\mathbf{T}=(\id+T_{1}\mathbf{T}^{-1}_{0})\mathbf{T}_{0}$ on $D(\mathbf{T})$. Since $\|T_1\mathbf T_0^{-1}\|<1$, the operator $\mathbf T$ admits a bounded inverse on $L^{1}(SM)$, this is given by $\mathbf{T}^{-1}=\mathbf{T}^{-1}_{0}(\id+T_{1}\mathbf{T}^{-1}_{0})^{-1}$. We have
\begin{multline*}
(\id+K)(\id-\mathbf{T}^{-1}T_{1})\\
=\id+\mathbf{T}^{-1}_{0}T_{1}-\mathbf{T}^{-1}_{0}(\id+T_{1}\mathbf{T}^{-1}_{0})^{-1}T_{1}
-\mathbf{T}^{-1}_{0}(T_{1}\mathbf{T}^{-1}_{0}(\id+T_{1}\mathbf{T}^{-1}_{0})^{-1})T_{1}.\\
= \id+\mathbf{T}^{-1}_{0}(\id-(\id+T_{1}\mathbf{T}^{-1}_{0})^{-1}-T_{1}\mathbf{T}^{-1}_{0}
(\id+T_{1}\mathbf{T}^{-1}_{0})^{-1})T_{1}=\id.
\end{multline*}
Similarly $(\id-\mathbf{T}^{-1}T_{1})(\id+K)=\id$. Therefore, $\id+K$ is invertible on $L^{1}(SM)$. Secondly, since $L^{1}(SM,\tau^{-1}d\Sigma^{2n-1})\subset L^{1}(SM)$, the operator $K$ is bounded on $L^{1}(SM,\tau^{-1}d\Sigma^{2n-1})$. If we can show $\mathbf{T}^{-1}T_{1}$ is bounded on $L^{1}(SM,\tau^{-1}d\Sigma^{2n-1})$, then $\id-\mathbf{T}^{-1}T_{1}$ will also be the inverse of $\id+K$ on subspace $L^{1}(SM,\tau^{-1}d\Sigma^{2n-1})$. That $\mathbf{T}^{-1}T_{1}$ is bounded on $L^{1}(SM,\tau^{-1}d\Sigma^{2n-1})$ follows from
\begin{multline*}
\|\tau^{-1}\mathbf{T}^{-1}T_{1}f|_{L^{1}(SM)}= \|\tau^{-1}\mathbf{T}^{-1}_{0}(\id+T_{1}\mathbf{T}^{-1}_{0})^{-1}T_{1}f\|_{L^{1}(SM)} \\
\le \|(\id+T_{1}\mathbf{T}^{-1}_{0})^{-1}T_{1}f\|_{L^{1}(SM)}\le \|(\id+T_{1}\mathbf{T}^{-1}_{0})^{-1}\|\|T_{1}f\|_{L^{1}(SM)}\\
\le \|(\id+T_{1}\mathbf{T}^{-1}_{0})^{-1}\|\|\tau\sigma_{p}\|_{L^{\infty}(SM)}\|\tau^{-1}f\|_{L^{1}(SM)},
\end{multline*}
where we have used Lemma \ref{T_0-inverse}, Lemma \ref{T1} and that the operator $\id+T_1\mathbf T_0^{-1}:L^1(SM)\to L^1(SM)$ has a bounded inverse. Thus, $\id-\mathbf{T}^{-1}T_{1}$ is the bounded inverse of $\id+K$ on $L^{1}(SM,\tau^{-1}d\Sigma^{2n-1})$. Therefore \eqref{part}, and hence \eqref{mag-tr-eq}, with $u|_{\p_+ SM}=u_+$ has the unique solution $u=(\id+K)^{-1}Ju_+$. To finish the proof of (a) it is left to show that $u\in \mathcal W$ whenever $u_+\in L^1(\p_+ SM)$. The proof of this and the proof of item (b) are exactly as in Proposition~\ref{well-posed-subcritical1}.
\end{proof}

\section{Distribution kernel of albedo operator}
This section is devoted to the derivation of distribution kernel of $\mathcal A$. To do this we need to solve \eqref{mag-tr-eq} with a singular boundary condition:
\begin{equation}\label{tr-bnd-sing}
\begin{aligned}
-\mathbf G_{\mu}u(x,\xi)-a(x,\xi)&u(x,\xi)+\int_{S_x M}k(x,\eta,\xi)u(x,\eta)\,d\eta_x=0,\\
&u|_{\p_+ SM}=\delta_{\{x',\xi'\}}(x,\xi),
\end{aligned}
\end{equation}
where $(x',\xi')\in\p_+ SM$ is regarded as a parameter and $\delta_{\{x',\xi'\}}$ is a distribution on $\p_+ SM$ defined by
$$
(\delta_{\{x',\xi'\}},\varphi)=\int_{\p_+ SM}\delta_{\{x',\xi'\}}(x,\xi)\varphi(x,\xi)\,d\mu(x,\xi)=\varphi(x',\xi') \quad\quad \varphi\in C^{\infty}_{c}(\partial_{+}SM).
$$

Since we assume that the magnetic system is simple, the following map is well-defined: for $(x,\xi)\in SM$ and for $y\in M$ we denote by $\mathcal P_x^y(\xi)$ the parallel translation of $\xi$ along the unique magnetic geodesic from $x$ to $y$.
\begin{Theorem}\label{u}
Assume that $(a,k)$ is admissible and the direct problem is well posed. Then the solution $u$ of \eqref{tr-bnd-sing} is of the form $u=u_1+u_2+u_3$, where
\begin{align}
u_1(x,\xi,x',\xi')&=\int_0^{\tau_+(x',\xi')}E(x,\xi,\tau_-(x,\xi),0)\delta_{(x,\xi)}(\phi_t(x',\xi'))\,dt,\label{u-1}\\
u_2(x,\xi,x',\xi')&=\int_0^{\tau_+(x',\xi')}\int^0_{\tau_-(x,\xi)}E(x,\xi,s,0) E(x',\xi',0,r)\cdot\label{u-2}\\
&\qquad\cdot k\left(\phi_r(x',\xi'),\mathcal P_{\gamma_{x,\xi}(s)}^{\gamma_{x',\xi'}(r)}(\dot\gamma_{x,\xi}(s))\right)\delta_{\gamma_{x,\xi}(s)}(\gamma_{x',\xi'}(r))\,ds\,dr,\notag\\
u_3(x,\xi,x',\xi')&\in L^{\infty}(\p_+ SM;\mathcal W),\label{u-3}
\end{align}
and where $(x,\xi)\in SM$ and $(x',\xi')\in \p_+SM$.
\end{Theorem}

\begin{proof}
To solve problem \eqref{tr-bnd-sing} take $\varphi_+\in C^{\infty}_{c}(\p_+ SM)$ and denote by $\varphi$ the solution to
\begin{equation}\label{tr-bnd-phi}
\begin{aligned}
-\mathbf G_{\mu}\varphi(x,\xi)-a(x,\xi)&\varphi(x,\xi)+\int_{S_x M}k(x,\eta,\xi)\varphi(x,\eta)\,d\eta_x=0,\\
&\varphi|_{\p_+ SM}=\varphi_+.
\end{aligned}
\end{equation}
As we explained in Section~\ref{3.1}, the equation \eqref{tr-bnd-phi} with $\varphi|_{\p_+ SM}=\varphi_+$ has the unique solution $\varphi=(\id+K)^{-1}J\varphi_+$, which can be rewritten as follows
$$
\varphi=J\varphi_+-K J\varphi_++(\id+K)^{-1}K^2 J\varphi_+.
$$
Since $\mathbf T^{-1}=(\id+K)^{-1}\mathbf T_0^{-1}$ and $K=\mathbf T_0^{-1} T_1$ we get
\begin{equation}\label{express}
\varphi=J\varphi_+-K J\varphi_++\mathbf T^{-1}T_1 K J\varphi_+.
\end{equation}
Now we analyze three terms in \eqref{express} for $\varphi$ and so we determine the distribution kernel of the solution operator $\varphi_+\mapsto \varphi$ of \eqref{tr-bnd-phi}. This solves \eqref{tr-bnd-sing} in the sense of distributions.

Let us begin with the first term $J\varphi_+$ on the right-hand side of \eqref{express}. If $\psi\in C^{\infty}_{c}(SM)$, then
\begin{multline*}
(J\varphi_+,\psi)\\
=\int_{\p_+ SM}\int_0^{\tau_+(x',\xi')}J\varphi_+(\phi_t(x',\xi'))\psi(\phi_t(x',\xi'))\,dt\,d\mu(x',\xi')\\
=\int_{\p_+ SM}\varphi_+(x',\xi')\int_0^{\tau_+(x',\xi')}E(\phi_t(x',\xi'),\tau_-(\phi_t(x',\xi')),0)\psi(\phi_t(x',\xi'))\,dt\,d\mu(x',\xi')\\
=\int_M\int_{S_x M}\int_{\p_+ SM}\varphi_+(x',\xi')\int_0^{\tau_+(x',\xi')}E(x,\xi,\tau_-(x,\xi),0)\\
\times\delta_{(x,\xi)}(\phi_t(x',\xi'))\,dt\,d\mu(x',\xi')\psi(x,\xi)\,d\sigma_x(\xi)\,d\vol_g(x),
\end{multline*}
where $\delta_{(\bar x,\bar\xi)}$ is the distribution on $SM$ defined by
$$
(\delta_{(\bar x,\bar\xi)},\varphi)=\int_M\int_{S_x M}\delta_{(\bar x,\bar\xi)}\varphi(x,\xi)\,d\sigma_x(\xi)\,d\vol_g(x)=\varphi(\bar x,\bar\xi) \quad\quad \varphi\in C^{\infty}_{c}(SM).
$$
Thus,
\begin{equation}\label{1st}
J\varphi_+(x,\xi)=\int_{\p_+ SM}u_1(x,\xi,x',\xi')\varphi(x',\xi')\,d\mu(x',\xi'),
\end{equation}
with $u_1$ as given in \eqref{u-1} -- the first statement of the theorem.

Consider further the second term $-KJ\varphi_+$ on the right-hand side of \eqref{express}. By straightforward calculations we get
\begin{multline*}
 -(KJ\varphi_+,\psi)=-\int_{\p_+ SM}\int_0^{\tau_+(x',\xi')}KJ\varphi_+(\phi_t(x',\xi'))\psi(\phi_t(x',\xi'))\,dt\,d\mu(x',\xi')\\
=\int_{\p_+ SM}\int_0^{\tau_+(x',\xi')}\int^0_{\tau_-(\phi_t(x',\xi'))}E(\phi_t(x',\xi'),s,0)(T_1 J\varphi_+)(\phi_{t+s}(x',\xi'))\\
\times\psi(\phi_t(x',\xi'))\,ds\,dt\,d\mu(x',\xi')\\
=\int_{M}\int_{S_x M}\int^0_{\tau_-(x,\xi)}\int_{S_{\gamma_{x,\xi}(s)}M}E(x,\xi,s,0)E(\gamma_{x,\xi}(s),\eta,\tau_-(\gamma_{x,\xi}(s),\eta),0)\\
\times k(\gamma_{x,\xi}(s),\eta,\dot\gamma_{x,\xi}(s))\varphi_+(\bar x,\bar \xi)\psi(x,\xi)\,d\sigma_{{\gamma_{x,\xi}(s)}}(\eta)\,ds\,d\sigma_x(\xi)\,d\vol_g(x),
\end{multline*}
where $(x,\xi)=\phi_t(x',\xi')$ and $(\bar x,\bar\xi)=\phi_{\tau_-(\gamma_{x,\xi}(s),\eta)}({\gamma_{x,\xi}(s)},\eta)$.

Using parallel translation along magnetic geodesics we obtain
\begin{multline*}
-(KJ\varphi_+,\psi)=\displaystyle\int_{M}\int_{S_x M}\int^0_{\tau_-(x,\xi)}\int_M \delta_{\gamma_{x,\xi}(s)}(y)\int_{S_y M}E(x,\xi,s,0)E(y,\eta,\tau_-(y,\eta),0)\\
\times k(y,\eta,\mathcal P_{\gamma_{x,\xi}(s)}^{y}(\dot\gamma_{x,\xi}(s)))\varphi_+(\phi_{\tau_-(y,\eta)}(y,\eta))\,\psi(x,\xi)\\
\times\,d\sigma_y(\eta)\,d\vol_g(y)\,ds\,d\sigma_x(\xi)\,d\vol_g(x),
\end{multline*}
where $\delta_x$ is a delta distribution on $M$ defined by
$$
(\delta_x,\varphi)_M=\int_M\delta_x(y)\varphi(y)\,d\vol_g(y)=\varphi(x) \quad\quad \varphi\in C^{\infty}_{c}(M).
$$
Make change of variables $(y',\eta',r)=(\phi_{\tau_-(y,\eta)}(y,\eta),-\tau_-(y,\eta))$. Interchanging the order of integration $\,dr\,d\mu(y',\eta')\,ds$ to $\,ds\,dr\,d\mu(y',\eta')$, we have
\begin{multline*}
-(KJ\varphi_+,\psi)=\displaystyle\int_{M}\int_{S_x M}\psi(x,\xi)\int_{\p_+ SM}\varphi_+(y',\eta')\int_0^{\tau_+(y',\eta')}\int^0_{\tau_-(x,\xi)}E(x,\xi,s,0) \\
\times E(y',\eta',0,r) k(\phi_r(y',\eta'),\mathcal P_{\gamma_{x,\xi}(s)}^{\gamma_{y',\eta'}(r)}(\dot\gamma_{x,\xi}(s))) \\
\times\delta_{\gamma_{x,\xi}(s)}(\gamma_{y',\eta'}(r))\,ds\,dr\,d\mu(y',\eta')\,d\sigma_x(\xi)\,d\vol_g(x)
\end{multline*}
and so obtain \eqref{u-2}.

Finally we analyse the third term in \eqref{express}. We start with $T_1 K J\varphi_+(x,\xi)$:
\begin{multline}\label{final-1}
T_1 K J\varphi_+(x,\xi)=\displaystyle\int_{S_x M}\int_{\tau_-(x,\eta)}^0 \int_{S_{\gamma_{x,\eta}(t)} M}k(x,\eta,\xi) k(\gamma_{x,\eta}(t),\zeta,\dot\gamma_{x,\eta}(t))E(x,\eta,t,0)\\
\times E(\gamma_{x,\eta}(t),\zeta,\tau_-(\gamma_{x,\eta}(t),\zeta),0)\varphi_+(\phi_{\tau_-(\gamma_{x,\eta}(t),\zeta)}(\gamma_{x,\eta}(t),\zeta))\\
\times\,d\sigma_{\gamma_{x,\eta}(t)}(\zeta)\,dt\,d\eta_x.
\end{multline}
We claim that the integral above is absolutely convergent for a.e. $(x,\xi)\in SM$. Indeed, as $K=\mathbf{T}^{-1}_{0}T_{1}$, from Lemma~\ref{T_0-inverse} and Lemma~\ref{T1} we have
\begin{multline*}
\|\tau^{-1}Kf\|_{L^{1}(SM)}=\|\tau^{-1}\mathbf{T}^{-1}_{0}T_{1}f\|_{L^{1}(SM)}\leq \|T_{1}f\|_{L^{1}(SM)}\\
\le \|\tau\sigma_{p}\|_{L^{\infty}(SM)}\|\tau^{-1}f\|_{L^{1}(SM)},
\end{multline*}
for $f\in L^1(SM)$. Using this and Lemma~\ref{T1} with Lemma~\ref{J is bounded}, we have
\begin{equation} \label{T_1K-bound}
\begin{aligned}
\|T_1 K J\varphi_+&\|_{L^1(SM)}\\
\le& \|T_1\tau\|_{L^1(SM)\to L^1(SM)} \|\tau^{-1}K\tau\|_{L^{1}(SM)\to L^1(SM)} \|\tau^{-1} J\varphi_+\|_{L^1(SM)}\\
\le& \|\tau\sigma_{p}\|^{2}_{L^{\infty}(SM)}\|\varphi_+\|_{L^1(\p_+ SM)}.
\end{aligned}
\end{equation}
Thus by Fubini's theorem, the integral $T_1 K J\varphi_+$ is absolutely convergent for a.e.  $(x,\xi)\in SM$, as claimed.

Make the following change of variables $y=\exp^{\mu}_x(t\eta)=\gamma_{x,\eta}(t)$ where $\exp^{\mu}_{x}$ is the magnetic exponential map. Let $J$ be the Jacobian of this change of variables. Introduce the following notations $\tilde\eta_x(y)=\dot\gamma_{x,\eta}(0)\in S_x M$, $\tilde\eta(y)=\dot\gamma_{x,\eta}(t)\in S_y M$. Then time $t$ can be expressed in terms of $x$ and $y$ as follows
$$
t(x,y)=-\frac{|(\exp^{\mu}_x)^{-1}y|}{|\tilde\eta_x(y)|}.
$$
Using all of this we obtain
\begin{multline*}
T_1 K J\varphi_+(x,\xi)=\int_M\int_{S_y M}k(x,\tilde\eta_x(y),\xi) k(y,\zeta,\tilde\eta_y(y))E(x,\tilde\eta_x(y),t(x,y),0)\\
\times E(y,\zeta,\tau_-(y,\zeta),0)\varphi_+(\phi_{\tau_-(y,\zeta)}(y,\zeta))|J|\,d\sigma_y(\zeta)\,d\vol_g(y).
\end{multline*}
Now, using Proposition \ref{Santalo} we change variables again: setting $(y,\zeta)=\phi_s(y',\zeta')$ we get
\begin{multline*}
T_1 K J\varphi_+(x,\xi)=\displaystyle\int_{\p_+ SM}\int_0^{\tau_+(x',\xi')}k(x,\tilde\eta_x(\gamma_{x',\xi'}(s)),\xi) k(\phi_s(x',\xi'),\tilde\eta(\gamma_{x',\xi'}(s)))\\
\times E(x,\tilde\eta_x(\gamma_{x',\xi'}(s)),t(x,\gamma_{x',\xi'}(s)),0) E(\phi_s(x',\xi'),\tau_-(\phi_s(x',\xi')),0)\\
\times \varphi_+(x',\xi'))|J|\,ds\,d\mu(x',\xi').
\end{multline*}
Rewriting this expression we obtain
$$
T_1 K J\varphi_+(x,\xi)=\int_{\p_+ SM}\tilde u_3(x,\xi,x',\xi')\varphi_+(x',\xi')\,d\mu(x',\xi'),
$$
where
\begin{multline*}
\tilde u_3(x,\xi,x',\xi')=\int_0^{\tau_+(x',\xi')}k(x,\tilde\eta_x(\gamma_{x',\xi'}(s)),\xi) k(\phi_s(x',\xi'),\tilde\eta(\gamma_{x',\xi'}(s)))\\
\times E(x,\tilde\eta_x(\gamma_{x',\xi'}(s)),t(x,\gamma_{x',\xi'}(s)),0) E(\phi_s(x',\xi'),\tau_-(\phi_s(x',\xi')),0)|J|\,ds.
\end{multline*}
Inequality \eqref{T_1K-bound} implies that $T_1 KJ\varphi_+\in L^1(SM)$ for all $\varphi_+\in L^1(\p_+ SM,\,d\mu)$. So, above equation and the Riesz representation theorem for $L^1(SM)$-valued functionals gives us that $\tilde u_3\in L^\infty(\p_+ SM;L^1(SM))$.

We write
$$
\mathbf T^{-1}T_1KJ\varphi_+(x,\xi)=\int_{\p_+ SM}u_3(x,\xi,x',\xi')\varphi_+(x',\xi')\,d\mu(x',\xi'),
$$
where $u_3=\mathbf T^{-1}\tilde u_3$. Our aim is to show that $u_3\in L^\infty(\p_+ SM;\mathcal W)$. By Proposition~\ref{prop} the operator $\tau^{-1}\mathbf T^{-1}$ is bounded from $L^1(SM)$ to $L^1(SM)$, so
$$
\tau^{-1}u_3\in L^\infty(\p_+ SM;L^1(SM)),
$$
or equivalently
$$
u_3\in L^\infty(\p_+ SM;L^1(SM,\tau^{-1}\,d\Sigma^{2n-1})).
$$
Note that we can write
$$
\mathbf G_\mu u_3=-a u_3+T_1 u_3+\tilde u_3.
$$
Since the map $u_3\mapsto -a u_3+T_1 u_3$ is bounded from $L^1(SM,\tau^{-1}\,d\Sigma^{2n-1})$ to $L^1(SM)$ by admissibility, we get that $\mathbf G_\mu u_3\in L^\infty(\p_+ SM; L^1(SM))$ which gives \eqref{u-3}.
\end{proof}

Now we describe the terms in the expansion of the distribution kernel of albedo operator $\mathcal A$.

\begin{Theorem}\label{kernel}
Let $(M,g,m)$ be a simple magnetic system. Assume that the pair $(a,k)$ is admissible and the direct problem is well posed. Then the albedo operator $\mathcal A:\p_+ SM \to \p_- SM$ is bounded and its Schwartz kernel $\alpha$ has the expansion $\alpha=\alpha_1+\alpha_2+\alpha_3$, where
\begin{align*}
\alpha_1(\hat x,\hat \xi,x',\xi')&=E(\hat x,\hat \xi,\tau_-(\hat x,\hat \xi),0)\delta_{\{\phi_{\tau_-(\hat x,\hat \xi)}(\hat x,\hat \xi)\}}(x',\xi'),\\
\alpha_2(\hat x,\hat \xi,x',\xi')&=\int_0^{\tau_+(x',\xi')}\int^0_{\tau_-(\hat x,\hat \xi)}E(\hat x,\hat \xi,s,0) E(x',\xi',0,r)\cdot\\
&\qquad\cdot k\left(\phi_r(x',\xi'),\mathcal P_{\gamma_{\hat x,\hat \xi}(s)}^{\gamma_{x',\xi'}(r)}(\dot\gamma_{\hat x,\hat \xi}(s))\right)\delta_{\gamma_{\hat x,\hat \xi}(s)}(\gamma_{x',\xi'}(r))\,ds\,dr,\notag\\
\alpha_3(\hat x,\hat \xi,x',\xi')&\in L^{\infty}(\p_+ SM;L^1(\p_- SM)).
\end{align*}
\end{Theorem}

\begin{proof}
Actually we have $\alpha_i(\cdot,\cdot,x',\xi')=u_i(\cdot,\cdot,x',\xi')|_{\p_- SM}$, $i=1,2,3$. Let $\varphi_+\in C^\infty_{c}(\p_+ SM)$, then for $(\hat x,\hat \xi)\in \p_- SM$ consider
$$
\varphi_i(\hat x,\hat \xi)=\int_{\p_+ SM}u_i(\hat x,\hat \xi,x',\xi')\varphi_+(x',\xi')\,d\mu(x',\xi').
$$
From \eqref{u-1} we have
$$\varphi_1(\hat x,\hat \xi)=\int_{\p_+ SM}\alpha_1(\hat x,\hat \xi,x',\xi')\varphi_+(x',\xi')\,d\mu(x',\xi')
$$
which proves the expression for $\alpha_1$.

The expression for $\alpha_2$ follows by \eqref{u-2} by setting $(x,\xi)=(\hat x,\hat \xi)\in \p_- SM$. Statement for $\alpha_3$ is true due to \eqref{u-3} and Proposition \ref{J}
\end{proof}

\section{Extracting from the kernel of albedo operator}

Let $d_\mu(x,y)$ be the length of the unit speed magnetic geodesic issued from $x$ and passing through $y$. Let $\psi\in C^\infty_{c}(\mathbb R)$ be such that $0\le \psi\le 1$, $\psi(0)=1$ and $\int_{\mathbb R} \psi(l)\,dl=1$. Define $\psi_{\varepsilon}(l)=\psi(l/{\varepsilon})$.

\begin{Proposition}\label{attenuation}
Let $(M,g,m)$ be a simple magnetic system. Assume that the pair $(a,k)$ is admissible and the direct problem is well posed. Then the following holds in $L^1_\mathrm{loc}(\p_- SM)$
$$
\lim_{\varepsilon\to 0}\int_{\p_+ SM}\alpha(x,\xi,x',\xi')\psi_{\varepsilon}(d_\mu(x',\gamma_{x,\xi}(\tau_-(x,\xi))))\,d\mu(x',\xi')=E(x,\xi,\tau_-(x,\xi),0).
$$
\end{Proposition}

\begin{proof}
When $\alpha$ is replaced by $\alpha_1$ the result is immediate. We need to show that when $\alpha$ is replaced by $\alpha_2$ and $\alpha_3$ the limits vanish. We consider $\alpha_2$ first.
\begin{multline*}
0\le \int_{\p_- SM}\int_{\p_+ SM}\int_0^{\tau_+(x',\xi')}\int_{\tau_-(x,\xi)}^0 k(\phi_r(x',\xi'),\mathcal P_{\gamma_{x,\xi}(s)}^{\gamma_{x',\xi'}(r)}(\dot\gamma_{x,\xi}(s)))\cdot\\
\times\delta_{\gamma_{x,\xi}(s)}(\gamma_{x',\xi'}(r))\psi_{\varepsilon}\circ d_\mu(x',\gamma_{x,\xi}(\tau_-(x,\xi)))\,ds\,dr\, d\mu(x',\xi')\,d\mu(x,\xi)\\
=\int_M\int_{S^2_y M}k(y,\eta',\eta)\psi_{\varepsilon}\circ d_\mu(\gamma_{y,\eta'}(\tau_-(y,\eta')),\gamma_{y,\eta}(\tau_-(y,\eta)))\,d\sigma_y(\eta')\,d\sigma_y(\eta)\,d\vol_g(y)
\end{multline*}
where $(y,\eta')=\phi_r(x',\xi')$ and $(y,\eta)=\phi_s(x,\xi)$. There exists a constant $C$ such that
\begin{multline*}
\supp \psi_\varepsilon\circ d_\mu(\gamma_{y,\eta'}(\tau_-(y,\eta')),\gamma_{y,\eta}(\tau_-(y,\eta)))\\
\subset W_\varepsilon=\{(y,\eta',\eta)\in S^2 M:\rVert\eta'-\eta\rVert_g<C\varepsilon\}.
\end{multline*}
where $S^{2}M:=\{(x,\xi,\eta):x\in M, \xi,\eta\in S_{x}M\}$ Then from above inequality we have
\begin{multline*}
0\le \int_{\p_-SM}\int_{\p_+ SM}\alpha_2(x,\xi,x',\xi') \psi_{\varepsilon}\circ d_\mu(x',\gamma_{x,\xi}(\tau_-(x,\xi)))\,d\mu(x',\xi')\,d\mu(x,\xi)\\
\le\int_{W_\varepsilon}k(y,\eta',\eta)\,d\sigma_y(\eta')\,d\sigma_y(\eta)\,d\vol_g(y)
\end{multline*}
which tends to zero as $\varepsilon\to 0$ since $k\in L^1(S^2 M)$ and the measure of $W_\varepsilon\to 0$ as $\varepsilon\to 0$.

Finally, for $\alpha_3$
\begin{multline*}
0\le\int_{\p_- SM}\left|\int_{\p_+ SM}\alpha_3(x,\xi,x',\xi')\psi_{\varepsilon}\circ d_\mu(x',\gamma_{x,\xi}(\tau_-(x,\xi)))\,d\mu(x',\xi')\right|\,d\mu(x,\xi)\\
\le\int_{V_\varepsilon}|\alpha_3(x,\xi,x',\xi')|\,d\mu(x',\xi')\,d\mu(x,\xi)\to 0
\end{multline*}
as $\varepsilon\to 0$. Here
\begin{multline*}
\supp \psi_\varepsilon\circ d_\mu(x',\gamma_{x,\xi}(\tau_-(x,\xi)))\\
\subset V_\varepsilon=\{(x,\xi,x',\xi')\in \p_- SM\times\p_+ SM: d_\mu(x',\gamma_{x,\xi}(\tau_-(x,\xi)))<C\varepsilon\}
\end{multline*}
and the limit above holds since, by Theorem \ref{kernel}, $\alpha_3\in L^1(\p_- SM\times\p_+ SM)$ and the measure of $V_\varepsilon\to 0$ as $\varepsilon\to 0$.
\end{proof}

Let $\psi\in C^\infty_{c}(\mathbb R^n)$ be such that $0\le \psi\le 1$, $\psi(x)=0$ for $|x|>1$, $\int_{\mathbb R^n}\psi(x)\,dx=1$. Define $\psi_\varepsilon:S_y M\to \mathbb R$ as
$$
\psi_\varepsilon(\eta)=\frac{1}{\varepsilon^n}\psi\left(\frac{\eta}{\varepsilon}\right),\quad \varepsilon>0,\quad \eta\in S_y M.
$$
Recall that if $f:S_y M\to \mathbb R$ is continuous then
$$
\int_{S_y M}f(\eta'-\eta)\varphi_\varepsilon(\eta)\,d\sigma_y(\eta)\to f(\eta'), \quad \eta'\in S_y M
$$
as $\varepsilon\to 0$.

Fix $(y,\eta',\eta)\in S^2 M$. Now let
$$
Z=\left\{\gamma_{y,\tilde\eta}(\tau_-(y,\tilde\eta)):\tilde\eta\in\Span\{\eta,\eta'\}\right\}\subset\p M.
$$
We let $h_1$ be a defining function for the set $Z$ as follows. For $z\in \p M$ let $\gamma_z$ be the magnetic geodesic issued from $z$ and passing through $y$ such that $\gamma(0)=z$ and $\gamma(1)=y$. Denote by $\pi(z)$ the orthogonal projection of $\dot\gamma(1)$ onto $(\Span\{\eta,\eta'\})^\perp$. Now we define $h_1(z)=\rVert \pi(z)\rVert_g$. For $z\in\p M$, we have $h_1(z)=0$ if and only if $z\in Z$.

Let $\varphi\in C^\infty_{c}(\mathbb R)$ satisfying $0\le\varphi\le 1$, $\varphi(0)=1$, $\int_{\mathbb R}\varphi(x)\,dx=1$ and define $\varphi_\rho:\p M\to \mathbb R$
$$
\varphi_\rho(z)=\varphi\left(\frac{h_1(z)}{\rho}\right),\quad \rho>0,\quad z\in \p M.
$$

Define $y(s)=\gamma_{y,\eta}(s)$ and $b(s)=\mathcal P_y^{y(s)}(\eta')\in S_{y(s)}M$ for $\tau_-(y,\eta)\le s \le \tau_+(y,\eta)$. Denote $x^*=\gamma_{y,\eta'}(\tau_-(y,\eta'))$. Also, define $h(s)=\gamma_{y(s),b(s)}(\tau_-(y(s),b(s)))\in \p M$. Note that $h(0)=x^*$ and $h(s)\in Z$ for each $s$, so $h_1(h(s))=0$.

Denote $\eta^*=dh(s)/ds|_{s=0}\in S_{x^*}\p M$. Note that $\eta^*\ne 0$ since $\eta$ and $\eta'$ are linearly independent. Define $h_2:\p M\to \mathbb R$ by
$$
h_2(x')=\langle(\exp_{x^*}^\mu)^{-1}(x'),\eta^*\rangle_g.
$$
This function is zero only at $x'=x^*$. Moreover, $dh_2(h(s))/ds|_{s=0}=\rVert \eta^*\rVert^2_g\ne 0.$

Define $W:=\{(x,\eta,\xi)\in S^{2}M:\eta\neq\pm\xi\}$. Let $\chi\in C^\infty_{c}(\mathbb R)$ such that $0\le\chi\le 1$, $\chi(0)=1$, $\int_{\mathbb R}\chi(x)\,dx=\rVert\eta^*\rVert^2$ and define $\chi_\delta:\p M\to \mathbb R$
$$
\chi_\delta(x')=\frac{1}{\delta}\chi\left(\frac{h_2(x')}{\delta}\right),\quad \delta>0,\quad x'\in\p M.
$$

\begin{Proposition}\label{scattering}
Let $(M,g,m)$ be a simple magnetic system with $\dim M\ge 3$. Assume that the pair $(a,k)$ is admissible and the direct problem is well posed. If $(y,\eta,\eta')\in S^2M$ with $\eta=\eta'$, then
\begin{align*}
\lim_{\delta\to 0}\lim_{\rho\to 0}\lim_{\varepsilon\to 0}\int_{\p_+ SM}\psi_\varepsilon(\mathcal P^{x'}_y (\eta')-\xi')\chi_\delta(x')\varphi_\rho(x')\alpha(\phi_{\tau_+(y,\eta)}(y,\eta),x',\xi')\,d\mu(x',\xi')\\
=E(\phi_{\tau_+(y,\eta)}(y,\eta),-\tau_+(y,\eta),0) E(\phi_{\tau_-(y,\eta')}(y,\eta'),0,-\tau_-(y,\eta'))k(y,\eta',\eta),
\end{align*}
with the limit holding in $L^1_\mathrm{loc}(W)$.
\end{Proposition}

\begin{proof}
Replacing $\alpha$ by $\alpha_1$ we obtain a multiple of $\chi_\delta(\gamma_{y,\eta}(\tau_-(y,\eta)))$ that is zero for small enough $\delta$ unless $\eta'$ and $\eta$ linearly dependent.

Replacing $\alpha$ by $\alpha_2$ and making the change of variables $(\hat y,\hat\eta)=\phi_r(x',\xi')$
\begin{multline*}
I_1:=\int_{\p_+ SM}\psi_\varepsilon(\mathcal P^{x'}_y(\eta')-\xi')\chi_\delta(x')\varphi_\rho(x')\alpha_2(\phi_{\tau_+(y,\eta)}(y,\eta),x',\xi')\,d\mu(x',\xi')\\
=\int_{M}\int_{S_{\hat y} M}\psi_\varepsilon(\mathcal P_y^{\gamma_{\hat y,\hat \eta}(\tau_-(\hat y,\hat \eta))}(\eta')-\dot\gamma_{\hat y,\hat \eta}(\tau_-(\hat y,\hat \eta)))\chi_\delta(\gamma_{\hat y,\hat \eta}(\tau_-(\hat y,\hat \eta)))\\
\times \varphi_\rho(\gamma_{\hat y,\hat \eta}(\tau_-(\hat y,\hat \eta)))\displaystyle\int_{\tau_-(\hat y,\hat \eta)}^{\tau_+(\hat y,\hat \eta)}E(\phi_{\tau_+(y,\eta)}(y,\eta),s-\tau_+(y,\eta),0)\\
\times E(\phi_{\tau_-(\hat y,\hat\eta)}(\hat y,\hat\eta),0,-\tau_-(\hat y,\hat\eta))k(\hat y,\hat\eta,\dot y(s))\delta_{y(s)}(\hat y)\,ds\,d\sigma_{\hat y}(\hat\eta)\,d\vol_g(\hat y).
\end{multline*}
Changing the order of integration from $\,ds\,d\sigma_{\hat y}(\hat\eta)\,d\vol_g(\hat y)$ to $\,d\sigma_{\hat y}(\hat\eta)\,d\vol_g(\hat y)\,ds$ we obtain
\begin{multline*}
I_1=\int_{\tau_-(y,\eta)}^{\tau_+(y,\eta)}\int_{S_{y(s)} M}\psi_\varepsilon(b(s)-\hat \eta)\chi_\delta(\gamma_{y(s),\hat\eta}(\tau_-(y(s),\hat\eta)))\varphi_\rho(\gamma_{y(s),\hat\eta}(\tau_-(y(s),\hat\eta)))\\
\times E(\phi_{\tau_+(y,\eta)}(y,\eta),s-\tau_+(y,\eta),0)E(\phi_{\tau_-(y(s),\hat\eta)}(y(s),\hat\eta),0,-\tau_-(y(s),\hat\eta))\\
\times k(y(s),\hat\eta,\dot y(s))\,d\sigma_{y(s)}(\hat\eta)\,ds
\end{multline*}
which tends to the following as $\varepsilon\to 0$
\begin{multline*}
I_2:=\int_{\tau_-(y,\eta)}^{\tau_+(y,\eta)}\chi_\delta(h(s)) E(\phi_{\tau_+(y,\eta)}(y,\eta),s-\tau_+(y,\eta),0)\\
\times E(h(s),\dot \gamma_{y(s),b(s)}(\tau_-(y(s),b(s))),0,-\tau_-(y(s),b(s))) k(y(s),b(s),\dot y(s))\,ds,
\end{multline*}
where we used that $\varphi_\rho(h(s))=1$. Perform the change of variable $s$ to $\tilde s=h_2(h(s))$ in above integral. This is possible since for sufficiently small $\delta$ for all $s$ in $\supp \chi_\delta(h(s))$ the following holds $\frac{d\tilde s}{ds}\ne 0$.

So, for sufficiently small $\tilde s_0$,
\begin{multline*}
I_2=\int_{-\tilde s_0}^{\tilde s_0}\frac{1}{\delta}\chi\left(\frac{\tilde s}{\delta}\right) E(\phi_{\tau_+(y,\eta)}(y,\eta),s(\tilde s)-\tau_+(y,\eta),0)\\
\times E(h(s(\tilde s)),\dot \gamma_{y(s(\tilde s)),b(s(\tilde s))}(\tau_-(y(s(\tilde s)),b(s(\tilde s)))),0,-\tau_-(y(s(\tilde s)),b(s(\tilde s))))\\
\times k(y(s(\tilde s)),b(s(\tilde s)),\dot y(s(\tilde s)))\frac{ds}{d\tilde s}\,d\tilde s
\end{multline*}
tends to the next value as $\delta\to 0$
$$
E(\phi_{\tau_+(y,\eta)}(y,\eta),-\tau_+(y,\eta),0) E(\phi_{\tau_-(y,\eta')}(y,\eta'),0,-\tau_-(y,\eta'))k(y,\eta',\eta).
$$

Finally, let $\beta(y,\eta,\eta')\in C^\infty_{c}(W)$. Then
\begin{multline}\label{lastline}
\int_M\int_{S_y M}\int_{S_y M}\bigg|\int_{\p_+ SM}\psi_\varepsilon(\mathcal P^{x'}_y (\eta')-\xi')\chi_\delta(x')\varphi_\rho(x')\\
 \times\beta(y,\eta,\eta')\alpha_3(\phi_{\tau_+(y,\eta)}(y,\eta),x',\xi')\,d\mu(x',\xi')\bigg|\,d\sigma_y(\eta')\,d\sigma_y(\eta)\,d\vol_g(y)\\
\le \frac{1}{\varepsilon\delta}\int_{\p_- SM}\int_{\tau_-(x,\xi)}\int_{S_{y(s)}M}\int_{\p_+ SM}\varphi_\rho(x')\beta(\phi_s(x,\xi),\eta')\\
\times\alpha_3(x,\xi,x',\xi')\,d\mu(x',\xi')\,d\eta'_{\gamma_{x,\xi}(s)}\,ds\,d\mu(x,\xi)\\
\le-\frac{1}{\varepsilon\delta}\int_{\p_+ SM}\int_{\p_- SM}\varphi_\rho(x') C_\beta(x,\xi)\tau_-(x,\xi)\\
\times\alpha_3(x,\xi,x',\xi')\,d\mu(x,\xi)\,d\mu(x',\xi')
\end{multline}
where
$$
C_\beta(x,\xi)=\sup_{s\in[\tau_-(x,\xi),0]}\int_{S_{y(s)}M}\beta(\phi_s(x,\xi),\eta')\,d\sigma_{\gamma_{x,\xi}(s)}(\eta').
$$
The integral \eqref{lastline} tends to zero as $\rho\to 0$, since the support of $\varphi_\rho$ is a $(3n-2)$-dimensional variety in the $(4n-4)$-dimensional domain $\p_- SM\times\p_+ SM$.
\end{proof}

\section{Uniqueness and non-uniqueness results}
\begin{proof}[Proof of Theorem~\ref{thA}]
Taking the limit as in Proposition \ref{attenuation} we determine
$$
\int_{\tau_-(x,\xi)}^0 a(\gamma_{x,\xi}(t))\,dt
$$
for all $(x,\xi)\in \p_-SM$. Thus we know the magnetic ray transform of $a$ which is invertible by \cite{DPSU} for isotropic $a$.

Now, taking the limit as in Proposition~\ref{scattering}, we determine $k$ for all $(x,\eta',\eta)\in W$. Since collision kernel $k$ is an $L^1$ function, knowing $k$ in $W$ is equivalent to knowing~$k$.
\end{proof}

\begin{proof}[Proof of Theorem~\ref{thB}]
Sufficiency has been verified before the statement of the theorem, so we need only show necessity. Proposition \ref{attenuation} implies that for any $(x,\xi)\in\p_- SM$
$$
\int_{\tau_-(x,\xi)}^0 a(\phi_t(x,\xi))\,dt=\int_{\tau_-(x,\xi)}^0\tilde a(\phi_t(x,\xi))\,dt.
$$
Define
$$
v(x,\xi):=\int_{\tau_-(x,\xi)}^0 (a-\tilde{a})(\phi_t(x,\xi))\,dt,\quad (x,\xi)\in SM.
$$
Then $v$ satisfies the following conditions: $v|_{\p_+SM}=0$ and $v|_{\p_- SM}=0$. Thus $v=0$ for all $x\in\p M$ and for almost every $\xi\in S_x M$. Define a function $w$ as follows
$$
w(x,\xi):=\exp{v(x,\xi)},\quad (x,\xi)\in SM.
$$
Then the following holds
$$
\mathbf G_\mu\log w(x,\xi)=\mathbf G_\mu v(x,\xi)=(a-\tilde{a})(x,\xi)
$$
which proves the first part of \eqref{gauge-eq}.

By Proposition \ref{scattering} for any $(y,\eta',
\eta)\in W$ it holds that
\begin{multline*}
E(\phi_{\tau_+(y,\eta)}(y,\eta),-\tau_+(y,\eta),0) E(\phi_{\tau_-(y,\eta')}(y,\eta'),0,-\tau_-(y,\eta'))k(y,\eta',\eta)\\
=\tilde E(\phi_{\tau_+(y,\eta)}(y,\eta),-\tau_+(y,\eta),0) \tilde E(\phi_{\tau_-(y,\eta')}(y,\eta'),0,-\tau_-(y,\eta'))\tilde k(y,\eta',\eta).
\end{multline*}
It is easy to see that
$$
\tilde E(\phi_{\tau_+(y,\eta)}(y,\eta),-\tau_+(y,\eta),0)=E(\phi_{\tau_+(y,\eta)}(y,\eta),-\tau_+(y,\eta),0)\frac{1}{w(y,\eta)}
$$
and
$$
\tilde E(\phi_{\tau_-(y,\eta')}(y,\eta'),0,-\tau_-(y,\eta'))=E(\phi_{\tau_-(y,\eta')}(y,\eta'),0,-\tau_-(y,\eta')) w(y,\eta').
$$
These give us the second part of \eqref{gauge-eq}.
\end{proof}

\begin{proof}[Proof of Theorem~\ref{thC}]
Interchanging $\eta'$ and $\eta$ in the second part of \eqref{gauge-eq} we get
\begin{align*}
\tilde k(y,\eta',\eta)&=\frac{w(y,\eta)}{w(y,\eta')}k(y,\eta',\eta)=\frac{w(y,\eta)}{w(y,\eta')}k(y,\eta,\eta')\\
&=\left(\frac{w(y,\eta)}{w(y,\eta')}\right)^2\tilde k(y,\eta,\eta')=\left(\frac{w(y,\eta)}{w(y,\eta')}\right)^2\tilde k(y,\eta',\eta).
\end{align*}
By positivity of $k$ we get that $w(y,\eta)=w(y,\eta')$ for all $\eta,\eta' \in S_y M$ which gives that $w=w(y)$ is independent of $\eta$ and $\tilde k=k$ by the second part of \eqref{gauge-eq}.

Using the first part of \eqref{gauge-eq} we obtain
$$
\tilde{a}(y,\eta)-a(y,\eta)=\mathbf G_\mu \log w(y)
$$
which proves item (i).

Next, item (ii) is proved since
$$
\tilde{a}(y,\eta)-a(y,\eta)=\tilde{a}(y,-\eta)-a(y,-\eta)=-\mathbf G_\mu \log w(y)
$$
combined with the previous equality gives $\tilde a=a$.
\end{proof}


\section{Statement of the main result on stability estimate}
In our next result we show stability of the gauge equivalent classes. The method of proof requires the total travel time to be strictly greater than zero for all magnetic geodesics in $M$. This can be done, without loss of generality, if we extend metric $g$ and magnetic potential $m$ to $g_1$ and $m_1$ respectively on a slightly larger manifold $M_1$, with $M\subset M_1^\textrm{int}$, in such a way that $(M_1,g_1,m_1)$ remains a simple magnetic system. As in \cite{MST2} we reduce the inverse problem on $M$ to the inverse problem on $M_1$, where attenuation coefficients and scattering coefficients are extended as zero in $M_1\setminus M$.

Let $(a,k)$ and $(\tilde{a},\tilde{k})$ be coefficients such that the direct problems in $(M,g,m)$ are well posed, and $\mathcal{A}$ and $\tilde{\mathcal{A}}$ the corresponding albedo operators. Define $a,\tilde{a},k,\tilde{k}$ to be zero in $M_{1}\backslash M$, then the forward problems in $(M_{1},g_{1},m_{1})$ are also well posed and the albedo operators $\mathcal{A}_{1}$ and $\tilde{\mathcal{A}}_{1}$ are well posed maps between function spaces on
$$
\p_\pm SM_1=\{(x,\xi)\in \p(SM_1):\pm\langle\xi,\nu_1(x)\rangle>0\}
$$
where $\nu_1$ is the inward unit normal to $\p M_1$. We can introduce a volume form $d\mu_1$ on $\p_\pm SM_1$ in a similar way as we introduced $d\mu$ on $\p_\pm SM$ in Section~\ref{3.1}. As in \cite{MST2}, when the direct problems for $M$ are well posed in $L^p$, $1\le p\le \infty$, the following holds
$$
\|\mathcal{A}-\tilde{\mathcal{A}}\|_{\mathcal{L}(L^{1}(\partial_{+}SM);L^{1}(\partial_{-}SM))}
=\|\mathcal{A}_{1}-\tilde{\mathcal{A}}_{1}\|_{\mathcal{L}(L^{1}(\partial_{+}SM_{1});L^{1}(\partial_{-}SM_{1}))}.
$$
This extension process allows to consider the problem on $(M_{1},g_{1},m_{1})$ where the coefficients are compactly supported without changing stability. Therefore, from now on we assume the coefficients are originally compactly supported on $M$ and define the quantity
\begin{equation} \label{c_0}
c_{0}:=\inf\{\tau(x,\xi):(x,\xi)\in SM\}>0.
\end{equation}

To state the stability result, we introduce some notations: define two norms for scattering coefficients
$$
\begin{array}{rcl} \vspace{1ex}
\|k\|_{\infty,1} &=& \|\sigma_{p}\|_{L^{\infty}(SM)}=\textrm{ess sup}_{(x,\eta)\in SM} \displaystyle\int_{S_{x}M}|k(x,\eta,\xi)|\,d\sigma_x(\xi),\\
\|k\|_{1} &=& \displaystyle\int_{M}\int_{S_{x}M}\int_{S_{x}M}|k(x,\eta,\xi)|\,d\sigma_x(\eta)\,d\sigma_x(\xi)\,d\vol_g(x).
\end{array}
$$
Denote by $\langle a,k\rangle$ the gauge equivalence class in the sense of \eqref{gauge-eq} which contains the pair $(a,k)$; the distance $\Delta$ between equivalence classes is defined to be the infimum of the distances between all possible pairs of representatives, i.e.
$$
\Delta(\langle a,k\rangle,\langle\tilde{a},\tilde{k}\rangle)=\inf_{(a',k')\in \langle a,k\rangle,\,(\tilde{a}',\tilde{k}')\in \langle\tilde{a},\tilde{k}\rangle}\max\{\|a'-\tilde{a}'\|_{L^{\infty}(SM)},\|k'-\tilde{k}'\|_{1}\}.
$$
For $\Sigma,\rho>0$, define the class
$$
\mathcal{U}_{\Sigma,\rho}:=\{(a,k) \textrm{ is admissible}:\|a\|_{L^{\infty}(SM)}\leq\Sigma,\|k\|_{\infty,1}\leq\rho\}.
$$

\begin{Theorem}\label{thE}
Let $(M,g,m)$ be a simple magnetized Riemannian manifold of dimension $n\geq 3$. Let $(a,k),(\tilde{a},\tilde{k})\in\mathcal{U}_{\Sigma,\rho}$ be such that they satisfy either \eqref{subcritical1} or \eqref{subcritical2}. Then there exists a constant $C$ depending on $\Sigma,\rho,n$ and $c_{0}$ in \eqref{c_0} such that
$$
\Delta(\langle a,k\rangle,\langle\tilde{a},\tilde{k}\rangle)\leq C\|\mathcal{A}-\tilde{\mathcal{A}}\|_{\mathcal{L}(L^{1}(\partial_{+}SM);L^{1}(\partial_{-}SM))}.
$$
More precisely, there exists a representative $(a',k')\in \langle a,k\rangle$ such that
\begin{equation} \label{stab-a}
\|a'-\tilde{a}\|_{L^{\infty}(SM)} \leq C\|\mathcal{A}-\tilde{\mathcal{A}}\|_{\mathcal{L}(L^{1}(\partial_{+}SM);L^{1}(\partial_{-}SM))},
\end{equation}
\begin{equation} \label{stab-k}
\|k'-\tilde{k}\|_{1} \leq C\|\mathcal{A}-\tilde{\mathcal{A}}\|_{\mathcal{L}(L^{1}(\partial_{+}SM);L^{1}(\partial_{-}SM))}.
\end{equation}
\end{Theorem}
This generalizes previously known results in Euclidean and Riemannian cases \cite{MST2,MST3}.

\section{Preliminary estimates}
In this section, as a preparation, we prove some preliminary estimates.
\begin{Proposition}\label{exist-1-map}
There is a family of maps $w_{\varepsilon,x'_0,\xi'_0}\in L^1(\p_+ SM)$, for $(x'_0,\xi'_0)\in\p_+SM$ and $\varepsilon>0$, such that $\|w_{\varepsilon,x'_0,\xi'_0}\|_{L^1(\p_+SM)}=1$ and for any $f\in L^\infty(\p_+ SM)$ we have
\begin{equation}\label{exist-1-map-int}
\lim_{\varepsilon\to0}\int_{\p_+ SM}w_{\varepsilon,x'_0,\xi'_0}(x',\xi')f(x',\xi')\,d\mu(x',\xi')=f(x'_0,\xi'_0)
\end{equation}
whenever $(x'_0,\xi'_0)$ is in the Lebesgue set of $f$. In particular, \eqref{exist-1-map-int} holds for almost every $(x'_0,\xi'_0)\in \p_+ SM$.
\end{Proposition}

\begin{proof}
For $(x'_0,\xi'_0)\in\p_+SM$ and $\varepsilon>0$ sufficiently small, let $(x',\xi'):U\times V\subset\mathbb R^{n-1}\times\mathbb R^{n-1}\to \p(SM)$ be a coordinate chart near $(x'_0,\xi'_0)=(x'(0),\xi'(0))$. Let $\,d\Sigma^{2n-2}(x',\xi')=\,d\Sigma^{2n-2}(x'(u),\xi'(u,v))=\sqrt{\det g}\,dv\,du$ be the expression for the volume form in this coordinate chart. For $(x',\xi')\in\p_+SM$, define
$$
w_{\varepsilon,x'_0,\xi'_0}(x',\xi')=\frac{1}{|\langle\nu(x'),\xi'\rangle|\sqrt{\det g}}\varphi_\varepsilon(u(x'))\varphi_\varepsilon(v(x',\xi'))
$$
where $\varphi(u)=1/\omega_{n-1}$ for $|u|<1$, $\varphi(u)=0$ for $|u|\ge1$, and $\varphi_\varepsilon(u)=\varphi(u/\varepsilon)/\varepsilon^{n-1}$. By $\omega_{n-1}$ we denote the volume of the unit ball in $\mathbb R^{n-1}$. Then for any $\varepsilon>0$
$$
\int_{\mathbb R^{n-1}}\varphi_\varepsilon(u)\,du=1.
$$
With above coordinates
$$
d\mu(x',\xi')=|\langle\nu(x'),\xi'\rangle|\,d\Sigma^{2n-2}(x',\xi')=|\langle\nu(x'),\xi'\rangle|\sqrt{\det g}\,du\,dv,
$$
so we have
\begin{multline*}
\int_{\p_+SM}w_{\varepsilon,x'_0,\xi'_0}(x',\xi')f(x',\xi')\,d\mu(x',\xi')\\
=\int_{\p_+SM}\varphi_\varepsilon(u(x'))\varphi_\varepsilon(v(x',\xi'))f(x',\xi')\frac{1}{\sqrt{\det g}}\,d\Sigma^{2n-2}(x',\xi')\\
=\int_{\mathbb R^{n-1}\times\mathbb R^{n-1}}\varphi_\varepsilon(u)\varphi_\varepsilon(v)f(x'(u),\xi'(u,v))\,dv\,du.
\end{multline*}
Applying above results to the function $f\equiv 1$ we get
$$
\int_{\p_+SM}w_{\varepsilon,x'_0,\xi'_0}(x',\xi')\,d\mu(x',\xi')=1,
$$
the proof is completed by using \cite[Theorem 8.15]{F}.
\end{proof}

For $(x,\xi,\eta)\in S^{2}M$ we define
$$
F(x,\xi,\eta)=E(x,\xi,\tau_-(x,\xi),0)E(x,\eta,0,\tau_+(x,\eta))
$$
which is the total attenuation along the broken magnetic geodesic
$$
\gamma_{x,\xi,\eta}(t)=\begin{cases}
\gamma_{x,\xi}(t)&\text{ if }\tau_-(x,\xi)\le t\le 0,\\
\gamma_{x,\eta}(t)&\text{ if }0\le t\le \tau_+(x,\eta).
\end{cases}
$$
\begin{Theorem}\label{1st-estimate}
Let $(a,k)$, $(\tilde a,\tilde k)$ be admissible pairs. For almost every $(x'_0,\xi'_0)\in\p_+ SM$ the following estimates hold: For $n\ge 2$,
\begin{equation}\label{pre-1}
|E-\tilde E|(x'_0,\xi'_0,0,\tau_+(x'_0,\xi'_0))\le\|\mathcal A-\tilde{\mathcal A}\|.
\end{equation}
For $n>2$, with $y(r)=\gamma_{x'_0,\xi'_0}(r)$,
\begin{multline}\label{pre-2}
\int_0^{\tau_+(x'_0,\xi'_0)}\int_{S_{y(r)}M}F(y(r),\dot y(r),\eta)(k-\tilde k)(y(r),\dot y(r),\eta)\,d\eta_{y(r)}\,dr\le \|\mathcal A-\tilde{\mathcal A}\|\\
+\int_0^{\tau_+(x'_0,\xi'_0)}\int_{S_{y(r)}M}|F-\tilde F|(y(r),\dot y(r),\eta)\tilde k(y(r),\dot y(r),\eta)\,d\eta_{y(r)}\,dr.
\end{multline}
where $\|\cdot\|=\|\cdot\|_{\mathcal{L}(L^{1}(\partial_{+}SM);L^{1}(\partial_{-}SM))}$.
\end{Theorem}

\begin{proof}
Let $(x'_0,\xi'_0)\in\p_+SM$ and let $w_{\varepsilon,x'_0,\xi'_0}\in L^1(\p_+ SM)$ be a family of maps defined as in Proposition \ref{exist-1-map}. We write $w_\varepsilon$ instead of $w_{\varepsilon,x'_0,\xi'_0}$ to simplify notations without any misunderstanding, since $(x'_0,\xi'_0)$ is fixed.

Let $h\in L^\infty(\p_-SM)$ with $\|h\|_{L^{\infty}(\p_{-}SM)}\le1$. Since $\|w_\varepsilon\|_{L^1(\p_+SM)}=1$, the mapping properties of the albedo operator imply that
\begin{equation}\label{1}
\left|\int_{\p_- SM}h(\hat x,\hat\xi)[\mathcal A-\tilde{\mathcal A}]w_\varepsilon(\hat x,\hat\xi)\,d\mu(\hat x,\hat\xi)\right|\le\|\mathcal A-\tilde{\mathcal A}\|.
\end{equation}
Since the albedo operator can be decomposed into three terms, we evaluate each of the three terms in
$$
\int_{\p_- SM}h(\hat x,\hat\xi)[\mathcal A-\tilde{\mathcal A}]w_\varepsilon(\hat x,\hat\xi)\,d\mu(\hat x,\hat\xi).
$$

We evaluate the first term using the first formula in Theorem~\ref{kernel}
\begin{multline*}
\mathcal I_1(h,\varepsilon):=\int_{\p_- SM}h(\hat x,\hat\xi)[\mathcal A_1-\tilde{\mathcal A}_1]w_\varepsilon(\hat x,\hat\xi)\,d\mu(\hat x,\hat\xi)\\
=\int_{\p_+ SM}h(\phi_{\tau_+(x',\xi')}(x',\xi'))w_\varepsilon(x',\xi')\hspace{100pt}\\
\times\left[ E(x',\xi',0,\tau_+(x',\xi'))-\tilde E(x',\xi',0,\tau_+(x',\xi'))\right]\,d\mu(x',\xi').
\end{multline*}
Since the integrand above is in $L^\infty(\p_+ SM)$ we may apply \eqref{exist-1-map-int} to get the following for a.e. $(x'_0,\xi'_0)\in\p_+ SM$
\begin{multline}\label{i_1}
\mathcal I_1(h)(x'_0,\xi'_0):=\lim_{\varepsilon\to0}\mathcal I_1(h,\varepsilon)\\
=h(\phi_{\tau_+(x'_0,\xi'_0)}(x'_0,\xi'_0))\,\left[ E(x'_0,\xi'_0,0,\tau_+(x'_0,\xi'_0))-\tilde E(x'_0,\xi'_0,0,\tau_+(x'_0,\xi'_0))\right].
\end{multline}

Next, we evaluate the second term using the second formula in Theorem \ref{kernel}. Set $y(r)=\gamma_{x',\xi'}(r)$
\begin{multline*}
\mathcal I_2(h,\varepsilon):=\int_{\p_- SM}h(\hat x,\hat\xi)[\mathcal A_2-\tilde{\mathcal A}_2]w_\varepsilon(\hat x,\hat\xi)\,d\mu(\hat x,\hat\xi)\\
=\int_{\p_+ SM}w_\varepsilon(x',\xi')\int_0^{\tau_+(x',\xi')}\int_{S_{y(r)}M}h(\phi_{\tau_+(y(r),\eta)}(y(r),\eta))\\
\times\Big[F(y(r),\dot y(r),\eta)k(y(r),\dot y(r),\eta)-\tilde F(y(r),\dot y(r),\eta)\tilde k(y(r),\dot y(r),\eta)\Big]\\
\times\,d\sigma_{y(r)}(\eta)\,dr\,d\mu(x',\xi').
\end{multline*}
Applying \eqref{exist-1-map-int} we obtain for a.e. $(x'_0,\xi'_0)\in \p_+SM$ with $y(r)=\gamma_{x'_0,\xi'_0}(r)$
\begin{multline}\label{i_2}
\mathcal I_2(h)(x'_0,\xi'_0):=\lim_{\varepsilon\to0}\mathcal I_2(h,\varepsilon)\\
=\int_0^{\tau_+(x'_0,\xi'_0)}\int_{S_{y(r)}M}h(\phi_{\tau_+(y(r),\eta)}(y(r),\eta))\\
\times\Big[F(y(r),\dot y(r),\eta)k(y(r),\dot y(r),\eta)-\tilde F(y(r),\dot y(r),\eta)\tilde k(y(r),\dot y(r),\eta)\Big]\\
\times\,d\sigma_{y(r)}(\eta)\,dr.
\end{multline}
We can write $\mathcal I_2=\mathcal I_{2,1}+\mathcal I_{2,2}$ with
\begin{multline}\label{i_2,1}
\mathcal I_{2,1}(h)(x'_0,\xi'_0):=\int_0^{\tau_+(x'_0,\xi'_0)}\int_{S_{y(r)}M}h(\phi_{\tau_+(y(r),\eta)}(y(r),\eta))\\
\times F(y(r),\dot y(r),\eta)(k-\tilde k)(y(r),\dot y(r),\eta)\,d\sigma_{y(r)}(\eta)\,dr
\end{multline}
and
\begin{multline}\label{i_2,2}
|\mathcal I_{2,2}(h)(x'_0,\xi'_0)|\le\int_0^{\tau_+(x'_0,\xi'_0)}\int_{S_{y(r)}M}|F-\tilde F|(y(r),\dot y(r),\eta)\\
\times\tilde k(y(r),\dot y(r),\eta)\,d\sigma_{y(r)}(\eta)\,dr.
\end{multline}
Now, we consider the third term
\begin{multline*}
\mathcal I_3(h,\varepsilon):=\int_{\p_- SM}h(\hat x,\hat\xi)[\mathcal A_3-\tilde{\mathcal A}_3]w_\varepsilon(\hat x,\hat \xi)\,d\mu(\hat x,\hat \xi)\\
 =\int_{\p_+ SM}w_\varepsilon(x',\xi')\int_{\p_-SM}h(\hat x,\hat\xi)(\alpha_3-\tilde\alpha_3)(\hat x,\hat\xi,x',\xi')\,d\mu(\hat x,\hat \xi)\,d\mu(x',\xi').
\end{multline*}
By Theorem \ref{kernel} the integrand above is in $L^\infty(\p_+ SM)$, then using \eqref{exist-1-map-int} we have for a.e. $(x'_0,\xi'_0)\in\p_+ SM$
\begin{equation}\label{i_3}
\mathcal I_3(h)(x'_0,\xi'_0):=\lim_{\varepsilon\to0}\mathcal I_3(h,\varepsilon)=\int_{\p_-SM}h(\hat x,\hat\xi)(\alpha_3-\tilde\alpha_3)(\hat x,\hat\xi,x'_0,\xi'_0)\,d\mu(\hat x,\hat\xi).
\end{equation}

Taking limit $\varepsilon\to0$ in \eqref{1} we get
\begin{equation}\label{2}
|\mathcal I_1(h)(x'_0,\xi'_0)+\mathcal I_2(h)(x'_0,\xi'_0)|\le \|\mathcal A-\tilde{\mathcal A}\|+\mathcal I_3(|h|)(x'_0,\xi'_0)
\end{equation}
for a.e. $(x'_0,\xi'_0)\in\p_+SM$ and for any $h\in L^\infty(\p_-SM)$ with $\|h\|_{L^{\infty}(\p_{-}SM)}=1$.

To prove estimates \eqref{pre-1} and \eqref{pre-2} we will consider sequence of functions $h$ for each case. First, we prove the estimate \eqref{pre-1}. Let $h_m\in L^\infty(\p_-SM)$ be $1$ in a shrinking neighbourhood of $(\hat x_0,\hat\xi_0):=\phi_{\tau_+(x'_0,\xi'_0)}(x'_0,\xi'_0)$. Consider the set
$$
\mathcal O_m(\hat x_0,\hat\xi_0)=\{\xi\in S_{\hat x_0}M:\|\xi-\hat\xi_0\|_{g(\hat x_0)}<\frac{1}{m}\}
$$
and extend $h_m$ to all $S_{\hat x_0}M$ as characteristic function of $\mathcal O_m(\hat x_0,\hat\xi_0)$. Then extend $h_m$ to all $SM$ by
$$
h_m(x,\xi)=
\begin{cases}
0&\text{if}\quad d_{\mu}(x,\hat x_0)\ge\frac{1}{m}\\
h_m(\hat x_0,\mathcal P_x^{\hat x_0}(\xi))&\text{if}\quad d_{\mu}(x,\hat x_0)<\frac{1}{m}.
\end{cases}
$$
By \eqref{i_1} this gives that
$$
\mathcal I_1(h_m)(x'_0,\xi'_0)=E(x'_0,\xi'_0,0,\tau_+(x'_0,\xi'_0))-\tilde E(x'_0,\xi'_0,0,\tau_+(x'_0,\xi'_0))
$$
independent of $m$. From \eqref{i_2} we have $\lim_{m\to\infty}\mathcal I_2(h_m)=0$, since for any $r$ the support of $h_m(\phi_{\tau_+(y(r),\eta)}(y(r),\eta))$ in $\eta\in S_{y(r)}M$ shrinks to $\dot y(r)$. From \eqref{i_3} and the third part of Theorem \ref{kernel}, we have $\lim_{m\to\infty}I_3(|h_m|)=0$, since the support shrinks to one point.

Next, we prove the estimate \eqref{pre-2}. Let $\mathcal N_m(x'_0,\xi'_0)\subset M$ be the tubular neighbourhood of the magnetic geodesic $y(r)=\gamma_{x'_0,\xi'_0}(r)$, $0\le r\le \tau_+(x'_0,\xi'_0)$ of radius $\frac{1}{m}$. Define a sequence $h_m\in L^\infty(\p_- SM)$ in several steps. First, set $h_m(\hat x,\hat\xi)=0$ if $\hat x\in \mathcal N_m(x'_0,\xi'_0)$. Note that $\mathcal I_1(h_m)=0$ for all $m$. Second, for $(\hat x,\hat\xi)\in \p_- SM$ with $\hat x\notin\mathcal N_m(x'_0,\xi'_0)$ we set $h_m(\hat x,\hat\xi)=0$ if the magnetic geodesic $z(s)=\gamma_{\hat x,\hat\xi}(s)$, $\tau_-(\hat x,\hat\xi)\le s\le 0$ does not intersect $\mathcal N_m(x'_0,\xi'_0)$. Third, if $\gamma_{\hat x,\hat\xi}$ intersects $\mathcal N_m(x'_0,\xi'_0)$ let $0\le r(\hat x,\hat\xi)\le\tau_+(x'_0,\xi'_0)$ and $\tau_-(\hat x,\hat\xi)\le s(\hat x,\hat\xi)\le 0$ be such that
$$
d_\mu(y(r(\hat x,\hat\xi)),z(s(\hat x,\hat\xi)))=\min_{r,s}\{d_\mu(y(r),z(s))\}
$$
and define
$$
h_m(\hat x,\hat\xi)=\sgn\left[(k-\tilde k)\left(y(r(\hat x,\hat\xi)),\dot y(r(\hat x,\hat\xi)),\mathcal P_{z(s(\hat x,\hat\xi))}^{y(r(\hat x,\hat\xi))}\left(\dot z(s(\hat x,\hat\xi))\right)\right)\right].
$$
Note that if $(\hat x,\hat\xi)=\phi_{\tau_+(y(r),\eta)}(y(r),\eta)$ for some $\eta\in S_{y(r)}M$, that is the broken magnetic geodesic with $(x'_0,\xi'_0)$ as the beginning and with $(\hat x,\hat \xi)$ as the end. the value $h_m(\hat x,\hat\xi)$ takes the sign of $k-\tilde k$ at the point of scattering. Note that the support of $h_m$ shrinks to a negligible set in $\p_- SM$ as $m\to\infty$ since $n>2$.

Applying \eqref{2} to $h_m$ we get
$$
|\mathcal I_{2,1}(h_m)|(x'_0,\xi'_0)\le \|\mathcal A-\tilde{\mathcal A}\|+\mathcal I_3(|h_m|)(x'_0,\xi'_0)+|\mathcal I_{2,2}(h_m)|(x'_0,\xi'_0)
$$
since $\mathcal I_1(h_m)=0$. The support of $h_m$ shrinks to zero in $\p_- SM$ as $m\to\infty$, so for almost every $(x'_0,\xi'_0\in \p_+SM)$ we have $\lim_{m\to\infty}\mathcal I_3(|h_m|)(x'_0,\xi'_0)=0$. Using \eqref{i_2,1}, \eqref{i_2,2} and $|\mathcal I_{2,1}(h_m)|=\mathcal I_{2,1}(|h_m|)$ we get for almost every $(x'_0,\xi'_0)\in\p_+SM$
\begin{multline*}
\lim_{m\to\infty}\mathcal I_{2,1}(h_m)(x'_0,\xi'_0)\\
=\int_0^{\tau_+(x'_0,\xi'_0)}\int_{S_{y(r)}M}F(y(r),\dot y(r),\eta)(k-\tilde k)(y(r),\dot y(r),\eta)\,d\sigma_{y(r)}(\eta)\,dr\\
\le\|\mathcal A-\tilde{\mathcal A}\|+\displaystyle\int_0^{\tau_+(x'_0,\xi'_0)}\int_{S_{y(r)}M}|F-\tilde F|(y(r),\dot y(r),\eta)\\
\times\tilde k(y(r),\dot y(r),\eta)\,d\sigma_{y(r)}(\eta)\,dr,
\end{multline*}
which is the estimate \eqref{pre-2}.
\end{proof}

\section{Stability estimates}
With the above preliminary estimates we give the proof of Theorem~\ref{thE}. Take two pairs $(a,k),(\tilde a,\tilde k)\in U_{\Sigma,\rho}$ and denote
$$
\varepsilon:=\|\mathcal A-\tilde{\mathcal A}\|.
$$
We construct pair $(a',k')\in \langle a,k\rangle$ such that statements of the theorem hold. Define first the ``trial" gauge transformation
\begin{equation}\label{trial}
w(x,\xi):=\exp\left\{-\int_{\tau_-(x,\xi)}^0(\tilde a-a)(\phi_s(x,\xi))\,ds\right\} \quad\text{a.e.}\quad(x,\xi)\in SM.
\end{equation}
Then $w$ is positive, $w|_{\p_+SM}=1$, $\mathbf G_\mu w\in L^\infty(SM)$ and
\begin{equation}\label{tilde a=a-Dlogw}
\tilde a(x,\xi)=a(x,\xi)-\mathbf G_\mu \log w(x,\xi).
\end{equation}
However the function $w$ is not equal to $1$ on $\p_-SM$. We will construct gauge transform, i.e. a function $\tilde w\in L^\infty(SM)$ with $1/\tilde w,\mathbf G_\mu \tilde w\in L^\infty(SM)$ and $\tilde w|_{\p SM}=~1$. The constructed $\tilde w$ will give us $(a',k')$. More precisely, for almost every $(x,\xi)\in SM$, we define $\tilde w(x,\xi)$ by
$$
\log \tilde w(x,\xi)=\log w(x,\xi)+\frac{\tau_-(x,\xi)}{\tau(x,\xi)}\log w(\phi_{\tau_+(x,\xi)}(x,\xi)).
$$
Simplicity assumption implies that $\tilde w\in L^\infty(SM)$ and it is easy to see that $\tilde w|_{\p SM}=~1$. Since the functions $\tau(x,\xi)$ and $\log w(\phi_{\tau_+(x,\xi)}(x,\xi))$ are constant along the orbits of magnetic flow $\phi$, we have
$$
\mathbf G_\mu \log \tilde w(x,\xi)=\mathbf G_\mu\log w(x,\xi)-\frac{\log w(\phi_{\tau_+(x,\xi)}(x,\xi))}{\tau(x,\xi)}\in L^\infty(SM),
$$
where we have used $\mathbf G_\mu \tau_-(x,\xi)=-1$.

Now, define the pair $(a',k')\in\langle a,k\rangle$ by
\begin{equation}\label{def a'}
a'(x,\xi):=a(x,\xi)-\mathbf G_\mu\log\tilde w(x,\xi)\quad\mathrm{and}\quad k'(x,\xi',\xi):=\frac{\tilde w(x,\xi)}{\tilde w(x,\xi')}k(x,\xi',\xi).
\end{equation}
The albedo operator $\mathcal A'$, corresponding to $(a',k')$, is equal to $\mathcal A$, and therefore
$$
\|\mathcal A'-\tilde{\mathcal A}\|=\|\mathcal A-\tilde{\mathcal A}\|=\varepsilon.
$$

We need to compare the pair $(a',k')$ with $(\tilde a,\tilde k)$ to show that they satisfy \eqref{stab-a} and \eqref{stab-k}. We begin with estimating $\log w$ at $\p_-SM$ which will be useful in estimating $\tilde a-a'$. By \eqref{pre-1} we have for almost every $(x'_0,\xi'_0)\in \p_+ SM$
$$
|E-\tilde E|(x'_0,\xi'_0,0,\tau_+(x'_0,\xi'_0))\le\varepsilon.
$$
Denoting $(\hat x_0,\hat \xi_0)=\phi_{\tau_+(x'_0,\xi'_0)}(x'_0,\xi'_0)$ we may rewrite above inequality as
\begin{equation}\label{upper}
|E-\tilde E|(\hat x_0,\hat \xi_0,\tau_-(\hat x_0,\hat \xi_0),0)\le\varepsilon.
\end{equation}
Introduce the notation $\diam_\mu(M):=\|\tau\|_{L^\infty(SM)}$. By the Mean Value Theorem applied to $u\mapsto \exp(-u)$ we obtain
\begin{multline}\label{lower}
|E-\tilde E|(\hat x_0,\hat \xi_0,\tau_-(\hat x_0,\hat \xi_0),0)=\exp(-u_0)\left|\int^0_{\tau_-(\hat x_0,\hat\xi_0)}(\tilde a-a)(\phi_t(\hat x_0,\hat\xi_0))\,dt\right|\\
=\exp\left(-u_0\right)|\log w(\hat x_0,\hat\xi_0)|
\ge\exp(-\diam_\mu (M)\Sigma)|\log w(\hat x_0,\hat\xi_0)|
\end{multline}
where $u_0(\hat x_0,\hat\xi_0)$ is a value between next quantities
$$
-\int_{\tau_-(\hat x_0,\hat\xi_0)}a(\phi_t(\hat x_0,\hat\xi_0))\,dt\quad\text{and}\quad-\int_{\tau_-(\hat x_0,\hat\xi_0)}\tilde a(\phi_t(\hat x_0,\hat\xi_0))\,dt.
$$
Comparing \eqref{upper} and \eqref{lower} we get the following estimate for $\log w$
\begin{equation}\label{logw-estimate}
|\log w(\hat x,\hat\xi)|\le\exp(\diam_\mu(M)\Sigma)\varepsilon,\quad\text{for a.e.}\quad (\hat x,\hat\xi)\in\p_- SM.
\end{equation}

Now,we estimate the difference between $\tilde a$ and $a'$. Using \eqref{tilde a=a-Dlogw},\eqref{def a'} and \eqref{logw-estimate},  for almost every $(x,\xi)$ we get
\begin{multline}\label{prefinal-a}
|(\tilde a-a')(x,\xi)|=|(\tilde a-a)(x,\xi)+(a-a')(x,\xi)|\\
=|\mathbf G_\mu\log\tilde w(x,\xi)-\mathbf G_\mu\log w(x,\xi)|\\
\frac{\log w(\phi_{\tau_+(x,\xi)}(x,\xi))}{\tau(x,\xi)}\le\varepsilon\frac{\exp(\diam_\mu(M)\Sigma)}{\tau(x,\xi)}.
\end{multline}
Since the coefficients are supported away from $\p M$, without loss of generality, from \eqref{prefinal-a} we obtain
$$
\|\tilde a-a'\|_\infty\le\varepsilon\frac{\exp(\diam_\mu(M)\Sigma)}{c_0}
$$
with $c_0$ defined in \eqref{c_0} . Estimate \eqref{stab-a} is proven.

Next, we prove estimate \eqref{stab-k}. From now on we work with the case $n>2$. From \eqref{prefinal-a} we get
$$
|a'(x,\xi)|\le\varepsilon\frac{\exp(\diam_\mu(M)\Sigma)}{\tau(x,\xi)}+\Sigma.
$$

Let $(x'_0,\xi'_0)\in\p_+ SM$ and $0<r<\tau_+(x'_0,\xi'_0)$ and $\eta\in S_y M$ with $y=\gamma_{x'_0,\xi'_0}(r)$. Using it we obtain
\begin{equation}\label{f-low}
\begin{aligned}
|F'(y(r),\dot y(r),\eta)|&=E(x'_0,\xi'_0,0,r)E(y(r),\eta,0,\tau_+(y(r),\eta))\\
&\ge\exp\left\{-2\left(\varepsilon\exp(\diam_\mu(M)\Sigma)+\diam_\mu(M)\Sigma\right)\right\}.
\end{aligned}
\end{equation}
Using non-negativity of $\tilde a$ and $a'$ we have
\begin{multline}\label{prefinal-k}|\tilde F-F'|(y(r),\dot y(r),\eta)\\
\le|\tilde E-E'|(x'_0,\xi'_0,0,r)+|\tilde E-E'|(y(r),\eta,0,\tau_+(y(r),\eta))\\
\le\bigg|\displaystyle\int_0^r[\tilde a-a'](\phi_s(x'_0,\xi'_0))\,ds\bigg|+\bigg|\int_0^{\tau_+(y(r),\eta)}[\tilde a-a'](\phi_s(x'_0,\xi'_0))\,ds\bigg|\\
\le\varepsilon\exp(\diam_\mu(M)\Sigma)\left(\displaystyle\frac{r}{\tau(x'_0,\xi_0)}+\frac{\tau_+(y(r),\eta)}{\tau(y(r),\eta)}\right)\\
\le 2\varepsilon\exp(\diam_\mu(M)\Sigma)
\end{multline}
by \eqref{prefinal-a}. Substituting \eqref{f-low},\eqref{prefinal-k} and applying the assumption $\|\tilde k\|_{\infty,1}\le\rho$ to the estimate \eqref{pre-2} with respect to the pairs $(a',k')$ and $(\tilde a,\tilde k)$ we get
\begin{multline*}
\int_0^{\tau_+(x'_0,\xi'_0)}\int_{S_{y(r)}M}|k'-\tilde k|(y(r),\dot y(t),\eta)\,d\sigma_{y(r)}(\eta)\,dr\\
\le\varepsilon\left(1+2\diam_\mu(M)\rho\,\omega_{n-1}\exp(\diam_\mu(M)\Sigma)\right)\\
\times\exp\left\{2\diam_\mu(M)\left(\varepsilon\frac{\exp(\diam_\mu(M)\Sigma)}{c_0}+\Sigma\right)\right\}=\varepsilon\, C_1.
\end{multline*}
Integrating the above inequality in $(x'_0,\xi'_0)\in \p_+ SM$ with respect the measure $\,d\mu(x'_0,\xi'_0)$, we get
$$
\|\tilde k-k'\|_1\le \varepsilon\vol(\p M)\omega_{n-1}C_1
$$
and Theorem~\ref{thE} holds with $C=\max\{\vol(\p M)\omega_{n-1}C_1,\exp(\diam_\mu(M)\Sigma)/c_0\}$.


\begin{thebibliography}{aa}
\bibitem{Ain} G. Ainsworth, {\it The attenuated magnetic ray transform on surfaces}, Inverse Probl. Imaging. {\bf 7}, 2013, 27--46.

\bibitem{AS} D. V. Anosov, Y. G. Sinai,
{\it Certain smooth ergodic systems [Russian]}, Uspekhi Mat. Nauk {\bf 22} (1967), 107--172.

\bibitem{Ar} V. I. Arnold,
{\it Some remarks on flows of line elements and frames}, Sov. Math. Dokl. {\bf 2} (1961), 562--564.

\bibitem{ArG} V. I. Arnold, A. B. Givental,
{\it Symplectic Geometry}, Dynamical Systems IV, Encyclopaedia
of Mathematical Sciences, Springer Verlag, Berlin, 1990.

\bibitem{A} S. Arridge, {\it Optical tomography in medical imaging}, Inverse Problems, {\bf 15} (1999),
R41--R93, MR 1684463 (2000b:78023), Zbl 0926.35155.

\bibitem{B} G. Bal, {\it Radiative transfer equations with varying refractive index: A mathematical perspective}, IJ. Opt. Soc. Amer. A, {\bf 23} (2006), 1639--1644.

\bibitem{BJ} G. Bal, A. Jollivet, {\it Stability estimates in stationary inverse transport}, Inverse Problems, {\bf 25} (2009), 075010, 32 pp.

\bibitem{Bi} S. Bissonette, {\it Imaging through fog and rain}, Opt. Eng. {\bf 31} (1992), 1045--1052.

\bibitem{CS} M. Choulli, P. Stefanov, {\it An inverse boundary value problem for the stationary
 transport equation}, Osaka J. Math.{\bf 36}(1) (1999), 87–104, MR 1670750
(2000g:35228), Zbl 0998.35064.

\bibitem{DL} R. Dautray, J.-L. Lions, {\it Mathematical Analysis and Numerical Methods for Since and Technology}, Vol, {\bf 6}, Springer Verlag, Berlin, 1993.

\bibitem{DPSU} N. S. Dairbekov, G. P. Paternain, P. Stefanov, G. Uhlmann, {\it The boundary rigidity problem in the presence of a magnetic field}, Adv. Math., {\bf 216} (2007), 535-609.

\bibitem{DU} N. Dairbekov, G. Uhlmann {\it Reconstructing the Metric and Magnetic Field from the Scattering Relation}, Inverse Problems and Imaging, 4(2010), 397-409.

\bibitem{F}G. Folland,``Real analysis, Modern Techniques and Their Applications", John Wiley \& Sons, New York, 1984.

\bibitem{Ko} V. V. Kozlov,
{\it Calculus of variations in the large and classical mechanics}, Russian Math. Surveys {\bf 40} (1985), no. 2, 37--71.

\bibitem{MC} N. J. McCormick, {\it Recent developments in inverse scattering transport methods}, Trans.
Theory and Stat. Phys. {\bf 13} (1984), 15--28.

\bibitem{M} S. McDowall, {\it An inverse problem for the transport equation in the presence of a Riemannian metric}, Pacific Journal of Math., Vol. iii, No. i, 2004

\bibitem{M2} S. McDowall, {\it Optical tomography on simple Riemannian surfaces}, Comm. PDE {\bf 30} (2005), no. 7-9, 1379-1400.

\bibitem{MST} S. McDowall, P. Stefanov, A. Tamasan, {\it Gauge equivalence in stationary radiative transport through media with varying index of refraction}, Inverse Problems and Imaging, Volume 4, No. 1, 2010, 151–167

\bibitem{MST2} S. McDowall, P. Stefanov and A. Tamasan, {\it Stability of the gauge equivalent in stationary inverse transport}, Inverse Problems {\bf 26}(2010), 025006, 19pp.

\bibitem{MST3} S. McDowall, P. Stefanov and A. Tamasan, {\it Stability of the gauge equivalent classes in inverse stationary transport in refractive media }, Contemporary Mathematics {\bf 559}(2011), 85-100.

\bibitem{Mi}R. Michel, {\it Sur la rigidit\'e impos\'ee par la longueur des g\'eod\'esiques}, Invent. Math. 65 (1981)
71-83.

\bibitem{MK} M. Mokhtar-Kharroubi, {\it Mathematical Topics in Neutron Transport Theory}, World Scientific, Singapore, 1997.

\bibitem{N1} S. P. Novikov,
{\it Variational methods and periodic solutions of equations
of Kirchhoff type. II}, Functional Anal. Appl. {\bf 15} (1981), 263--274.

\bibitem{N2} S. P. Novikov,
{\it Hamiltonian formalism and a multivalued analogue of Morse theory}, Russian Math. Surveys {\bf 37} (1982), no. 5, 1--56.

\bibitem{NS} S. P. Novikov, I. Shmel'tser,
{\it Periodic solutions of the Kirchhoff equations for the free
motion of a rigid body in a liquid, and the extended
Lyusternik-Schnirelmann-Morse theory. I}. J. Functional Anal. Appl. {\bf 15} (1981), 197--207.

\bibitem{PP} G. P. Paternain, M. Paternain,
{\it Anosov geodesic flows and  twisted symplectic structures}, International Congress on Dynamical Systems in Montevideo (a tribute to Ricardo Ma\~n\'e), F. Ledrappier, J. Lewowicz, S. Newhouse eds,
Pitman Research Notes in Math. {\bf 362} (1996), 132--145.

\bibitem{RS} M. Reed, B. Simon, {\it Methods of Modern Mathematical Physics}, Vol. {\bf 3}, Academic Press, New York, 1979.

\bibitem{SGKZ} J. R. Singer, F. A. Gr\"unbaum, P. Kohn, and J. P. Zubelli, {\it Image reconstruction
of the interior of bodies that diffuse radiation}, Science {\bf 248} (1990), 990--993.

\bibitem{ST} P. Stefanov, A. Tamasan, {\it Uniqueness and non-uniqueness in inverse radiative transfer},  Proc. Amer. Math. Soc. 137(2009), 2335-2344.

\bibitem{SU} P. Stefanov, G. Uhlmann,
{\it An inverse source problem in optical molecular imaging}, Analysis \& PDE. {\bf 1}(1) (2008), 115--126.
\end{thebibliography}
\end{document}